\documentclass[leqno,12pt]{amsart}
\usepackage{amsaddr}
\usepackage{amssymb}
\usepackage{amscd}
\usepackage{mathrsfs}
\usepackage{enumitem}

\usepackage[all]{xy}            
\usepackage{array}

\newcommand{\Hilb}{\mathrm{Hilb}}
\newcommand{\ad}{\mathrm{ad}}
\newcommand{\supp}{\mathrm{Supp}}

\newtheorem{theorem}{Theorem}
\newtheorem{proposition}[theorem]{Proposition}
\newtheorem{lemma}[theorem]{Lemma}
\newtheorem{corollary}[theorem]{Corollary}
\theoremstyle{definition}
\newtheorem{definition}[theorem]{Definition}
\newtheorem{remark}[theorem]{Remark}
\newtheorem{claim}[theorem]{Claim}
\newtheorem{example}[theorem]{Example}
\newtheorem{conjecture}[theorem]{Conjecture}

\title{Wonderful Varieties: A geometrical realization}

\author[St\'ephanie Cupit-Foutou]{St\'ephanie Cupit-Foutou}
\address{Ruhr-Universit\"at Bochum \\ Fakult\"at f\"ur Mathematik \\
Raum NA 4/67\\
D-44780 Bochum\\ Tel. (++49) 234 32-23331\\ Fax (++49) 234 32-14498}
\email{stephanie.cupit@rub.de}
\thanks{This work was partially supported by the DFG Schwerpunktprogramm 1388 - Darstellungstheorie}

\begin{document}

\begin{abstract}
A geometrical realization of wonderful varieties by means of a suitably chosen class of invariant Hilbert schemes is given.
As a consequence, Luna's conjecture asserting that wonderful varieties are classified by combinatorial invariants, called spherical systems, is proved.
\end{abstract}

\bibliographystyle{amsalpha}

\maketitle

\tableofcontents

\section*{Introduction}

Wonderful varieties are smooth equivariant compactifications $X$ of spherical homogeneous spaces of a connected complex reductive algebraic group $G$.
A normal algebraic $G$-variety is spherical if it contains a dense orbit for a Borel subgroup of $G$.
If $X$ itself is $G$-homogeneous then it is a flag variety. If however, the boundary $D$ of the open $G$-orbit in $X$ is non-empty, then in order for $X$ to be wonderful, the $G$-variety $D$ must satisfy strong conditions (\textsl{see} Definition~\ref{defwonderful}).
These conditions were already transparent in the compactifications of symmetric spaces constructed and studied by DeConcini-Procesi in~\cite{DCP}.

Wonderful varieties are not only interesting in their own right, they also lend themselves to the study of the wider class of spherical varieties.
As for toric varieties, the equivariant embeddings of spherical $G$-homogeneous spaces (which are just tori in case of toric varieties) are classified by combinatorial objects, the so-called colored fans; \textsl{see}~\cite{LV,K91}.
This classification has been known since the early 80s. That concerning spherical $G$-homogeneous spaces has yet to be carried out.
Not all spherical $G$-homogeneous spaces can be compactified into a wonderful $G$-variety but whenever it is so, such a compactification is unique up to a $G$-isomorphism.
One of the most remarkable breakthroughs in the aforementioned classification problem was made by Luna (\cite{Lu01})
while proving that the classification of spherical homogeneous varieties reduces to that of wonderful varieties and proposing a framework to classify the latter.

Here, motivated by Luna's principle, we do indeed classify all wonderful varieties by verifying Luna's fundamental question (\textsl{see} below).

The unique closed $G$-orbit of a wonderful $G$-variety $X$, the weights w.r.t. a Borel subgroup $B$ of $G$ of the function field of $X$ as well as the $B$-stable prime divisors of $X$ are invariants of special interest; they are/yield some combinatorial invariants.
After Wasserman completed the classification of rank $2$ wonderful varieties (\cite{W}),
Luna highlighted in~\cite{Lu01} some properties enjoyed by the aforementioned combinatorial invariants of such varieties and took them as axioms to set up the definition of \textit{spherical systems of $G$}.
Further, Luna proved that any wonderful $G$-variety gave rise to a spherical system  of $G$.

Luna's conjecture asserts that for any spherical system there exists a unique wonderful variety producing this system.
It should be underlined that the uniqueness part of this conjecture can be derived from Losev's work (\cite{Lo}) for all groups $G$.
Further,  a number of special cases of this conjecture have been handled case by case using the procedure and initial work of Luna (\cite{Lu01}); \textit{see} \cite{BP, Bra,Lu07,BC2}.\footnote{While this paper was being revised, Bravi and Pezzini completed the existence part of this program in~\cite{BP14, BP11}.}
The approach followed therein is Lie theoretical: for a given spherical system of some group $G$, a subgroup $H$ of $G$ is exhibited and thereafter $G/H$ is proved to admit a wonderful compactification with the desired spherical system.

The approach adopted in this paper involves the invariant Hilbert schemes introduced by Alexeev and Brion in~\cite{AB}.
It not only leads to a proof of the complete conjecture (existence and uniqueness) but perhaps more importantly it gives a geometrical realization of the wonderful varieties at hand.

Besides the combinatorial invariants, let us mention the total coordinate ring (known also as the Cox ring) of a wonderful $G$-variety, an algebro-geometric invariant studied by Brion in~\cite{Br07}.
The structure of this ring gives insight to the aforementioned combinatorial invariants of a  wonderful variety.
Moreover, as shown in loc. cit.,
this ring is factorial and finitely generated; the spectrum
of this ring is the total space of a flat family of affine spherical $G$-varieties whose coordinate rings are isomorphic as $G$-modules.
As shown in~\cite{AB}, closed $G$-subschemes of a given finite dimensional $G$-module, whose coordinate ring has a prescribed structure of $G$-module
are parameterized by a quasi-projective scheme, the invariant Hilbert schemes; these schemes allow to prove several results concerning spherical varieties.
Connections between invariant Hilbert schemes and wonderful varieties were established in~\cite{Js,BC1};
the invariant Hilbert schemes considered there were proved to be affine spaces and the corresponding universal families
turn out to be the families occurring in Brion's work. These results were obtained by means of the
already known classification of wonderful varieties of rank $1$ (resp. of strict wonderful varieties) (\cite{A} and also~\cite{Br89-1}, resp. \cite{BC1}).

The problem of determining classification-free these invariant Hilbert schemes then arises naturally
(\textsl{see} also~\cite{Lu07} for related questions).
One of the main goals of this paper is to solve this problem for a class of invariant Hilbert schemes by deformation theoretic methods.
Once this is achieved, a construction of wonderful varieties is provided and  Luna's conjecture is proved.

We start by gathering basic material on wonderful varieties and spherical systems.
In the first section, we briefly recall the definition of wonderful varieties and of their invariants as well as some of their properties;
Luna's definition of spherical systems is stated in the second section.
For later purposes, to any spherical system of some given group $G$, we attach a set of characters 
$$
\lambda_D=(\omega_D,\chi_D) \quad\mbox{ indexed by a finite set $\Delta$.}
$$
The $\omega_D$'s are dominant weights of $G$ defined after~\cite{F} and the $\chi_D$'s are characters of some well-determined diagonalizable group $C$.
The characters $\omega_D$ (resp. $\chi_D$) encapsulate the first (resp. the third) datum of the spherical system under consideration.
We conclude the second section by recollecting how wonderful varieties and spherical systems are related and by stating explicitly Luna's conjecture.
Further, we give a geometrical interpretation of the characters $\lambda_D$.

The third section is devoted to definitions and results concerning invariant Hilbert schemes. 
In order to study later the geometry of these schemes, following~\cite{AB}, we give a description of their tangent spaces and define a toric action on these schemes.
Further, we  set up an obstruction theory for the functor of invariant infinitesimal deformations of the most degenerate point of a given invariant Hilbert scheme.
As an application, we obtain a smoothness criterion for the invariant Hilbert schemes under consideration (Corollary~\ref{Schlessinger}).

In the fourth section, we settle our main results.
We start by assigning an invariant Hilbert scheme to any given spherical system $\mathscr{S}$; this scheme is defined precisely up to the aforementioned weights $\lambda_D$ ($D\in \Delta$) associated to $\mathscr{S}$.
Many geometrical properties of wonderful varieties can be translated into combinatorial properties of their spherical systems and vice versa (\cite{Lu01}).
This provides a natural dictionary which in turn allows many reductions to prove Luna's conjecture.
Based on this fact, we consider only \emph{spherically closed} systems $\mathscr S$ (Definition~\ref{sphericallyclosed})
and we study the  geometry of the corresponding invariant Hilbert schemes or rather an open subsets $\Hilb(\mathscr{S})$ of them.
This study appeals to a precise description of the tangent space along with the aforementioned obstruction space and leads to the following theorem.

\begin{theorem}(Theorem~\ref{Hilb})
Let $\mathscr S$ be a spherically closed system of $G$.
The scheme $\Hilb(\mathscr{S})$ is isomorphic to an affine space where an adjoint maximal torus of $G$ acts linearly with
weights equal to the opposites of the spherical roots of $\mathscr{S}$.
\end{theorem}

Let thus $\mathbb A^r\simeq\mathrm{Hilb}(\mathscr{S})$.
Set
$$
\tilde G:=G\times C^\circ
$$
where $C^\circ$ denotes the identity-component of the aforementioned diagonalizable group $C$ associated to $\mathscr S$
and
$$
V=\oplus_{D\in\Delta} V(\lambda_D)^*
$$
where $V(\lambda_D)^*$ stands for the dual of the irreducible $\tilde G$-module with highest weight $\lambda_D$.

Consider the universal family of the functor represented by  $\mathrm{Hilb}(\mathscr{S})$
$$
\begin{CD}
V\times \mathbb A^r \supset\mathcal X^{\mathrm{univ}}  @>{\pi}>> \mathbb A^r
\end{CD}
$$
and let $\mathring{\mathcal X}^{\mathrm{univ}}$ be the open subset of $\mathcal X^\mathrm{univ}$ defined as follows
$$
\mathring{\mathcal X}^{\mathrm{univ}}=\{x\in\mathcal X^{\mathrm{univ}}: \tilde{G}.x \mbox{ is open in } \pi^{-1}\pi(x)\}.
$$
There is an action of the algebraic torus $\mathbb G_m^\Delta=GL(V)^{\tilde G}$ on $\mathcal X^{\mathrm{univ}}$.
This toric action stabilizes the set $\mathring{\mathcal X}^{\mathrm{univ}}$; \textsl{see} Section 5.2.

\begin{theorem}(Theorem~\ref{existence})
Let $\mathscr S$ be a spherically closed system of $G$.
The quotient
$$
X(\mathscr{S})=\mathring{\mathcal X}^{\mathrm{univ}}/ \mathbb G_m^\Delta
$$
exists and is a wonderful $G$-variety whose spherical system is the given $\mathscr{S}$.
Further its total coordinate ring is the coordinate ring of $\mathcal X^{\mathrm{univ}}$.
\end{theorem}

Combining the above result with Luna's reduction procedure, we conclude by proving Luna's conjecture (Corollary~\ref{conjecture}):

\begin{corollary}
The map $X\mapsto \mathscr S_X$, assigning to a wonderful $G$-variety its spherical system,
induces a bijective correspondence between the set of isomorphism classes of wonderful $G$-varieties and the set of spherical systems of $G$.
\end{corollary}

Appendix~\ref{Basicmaterialonsphericalvarieties} collects some already known results on spherical varieties.

In Appendix~\ref{combinatorics}, we recollect the list of spherical roots and we derive some significant combinatorial properties of the weights associated to any spherical system.

Appendix~\ref{cohomologies} (with the help of Appendix~\ref{combinatorics}) can be read independently to the body text.
Firstly, we perform computations on the cohomology spaces in degree $0$ and $1$ of the isotropy Lie algebra of a sum $v_\Delta$ of highest weight vectors of $V$ with coefficients in $V/\tilde{\mathfrak g}.v_\Delta$, $\tilde{\mathfrak g}$ being the Lie algebra of $\tilde G$. These spaces appear in the characterizations of the tangent space and the
obstruction space already mentioned. These results are thus applied in the last subsection to prove that this obstruction space is trivial.
\bigbreak

\paragraph{\textit{Acknowledgments}} I am very grateful to Michel Brion for the interest he demonstrated in my work, through his comments and advice.
I thank also Dominique Luna for enlightening exchanges, Peter Littelmann for the support he provides me and Dmitry Timashev for his critical reading of a previous version of this paper.

\bigbreak

\paragraph{\textbf{Notation}}

The ground field $k$ is the field of complex numbers.
Throughout this paper, $G$ is a connected reductive algebraic group.
We fix a Borel subgroup $B$ of $G$ and $T\subset B$ a maximal torus; the unipotent radical of $B$ is denoted by $U$.
The choice of $(B,T)$ defines the set of simple roots $S$ of $G$ as well as the set $\Lambda^+$ of dominant weights.
We label the simple roots as in Bourbaki (\cite{Bo}).

Let $\Xi(H)$ denote the character group of any group $H$; note that $\Xi(B)$ and $\Xi(T)$ are naturally identified.
Let $(\cdot,\cdot)$ be the natural pairing between $\Xi(T)$ and $\mathrm{Hom_\mathbb Z(\Xi(T),\mathbb Z)}$.
Then we have in particular that $(\alpha,\alpha^\vee)=2$ where $\alpha^\vee$ stands for the co-root of the simple root $\alpha$.
For any $\mu\in \Xi(T)$, $e^\mu$ refers to the corresponding regular function on $T$.

Recall that $\Lambda^+$ parametrizes the simple $G$-modules; by $V(\lambda)$, we denote the simple $G$-module
associated to $\lambda\in\Lambda^+$. The dual module $V(\lambda)^*$ is isomorphic to $V(\lambda^*)$ with $\lambda^*=-w_0(\lambda)$, $w_0$ being the longest element of the Weyl group of $(G,T)$.
Given any $G$-module $V$ and a weight $\mu\in\Xi(T)$ of $V$, let $V_\mu$ denote the $\mu$-weight space of $V$
and let $V_{(\mu)}$ be the isotypical component of type $V(\mu)$ in case $\mu\in\Lambda^+$.

For a given operation of a group $H$ on a set $X$, the set of points in $X$ which are fixed by $H$ is denoted by $X^H$.

\section{Wonderful varieties}\label{wondervar}

Throughout this section, we recall notions and results concerning wonderful varieties; for further details, unless otherwise stated, one may consult~\cite{Lu01} and the references given therein.

\begin{definition}\label{defwonderful}
A smooth complete algebraic variety equipped with an action of $G$ is said to be
\textit{wonderful} (of \textit{rank} $r$) if
\smallbreak
\noindent
{\rm (i)}\enspace
it contains a dense $G$-orbit whose complement is the union of $r$ smooth prime divisors $D_1,\ldots, D_r$ with normal crossings;
\smallbreak
\noindent
{\rm (ii)}\enspace
its $G$-orbit closures are given by the partial intersections $\cap_{i\in I} D_i$, where $I$ is a subset of $\{1,\ldots, r\}$.
\end{definition}

The radical of $G$ acts trivially on any wonderful $G$-variety.
Accordingly, we assume in the rest of this section that $G$ is semisimple and simply connected.

An algebraic $G$-variety is called \textsl{spherical} if it is normal and contains a dense $B$-orbit.
A spherical $G$-variety is called \textsl{toroidal} if each of its $B$-stable prime divisors which contains a $G$-orbit is also $G$-stable.

\begin{proposition}\label{w'fulcriterion}
A $G$-variety is wonderful if and only if it is complete, smooth, spherical, toroidal and contains a unique closed orbit of $G$
\end{proposition}

Let $X$ be a wonderful $G$-variety.
Denote the set of its $B$-stable and not $G$-stable prime divisors by $\mathcal D_X$.
We call $\mathcal D_X$ \textsl{the set of colors of $X$}; this set yields a basis of the Picard group $\mathrm{Pic}(X)$ of $X$ (\textsl{see}~\cite{Br89}).

Let $H\subset G$ be the stabilizer of a point in the open $G$-orbit of $X$.
Choose $H$ such that $BH$ is open in $G$ and let $p:G\rightarrow G/H$ be the natural projection.
For any $D\in\mathcal D_X$, $p^{-1}(D)$ is a $(B\times H)$-stable divisor w.r.t the left $B$-action and the right $H$-action on $G$.
Since $G$ is assumed to be factorial, $p^{-1}(D)$ is given by an equation.
Let $f_D\in k[G]$ be an equation of $p^{-1}(D)$; $f_D$ is uniquely defined by requiring that $f_D(1)=1$ and $f_D$ is a $(B\times H)$-eigenvector.
The weights $(\omega_D,\chi_D)$ of the $f_D$'s generate freely the abelian group $\Xi(B)\times_{\Xi(B\cap H)} \Xi(H)$;
\textsl{see} Lemma 2.2.1 in~\cite{Br07}.

Let $H^\sharp$ denote the intersection of the kernels of all characters of $H$.
Then $H^\sharp$ is a normal subgroup of $G$ and $H/H^\sharp$ is a diagonalizable group whose character group is $\Xi(H)$.
The variety $G/H^\sharp$ is quasi-affine and spherical under the natural action of $G\times (H/H^\sharp)^\circ$.

\subsection{Combinatorial invariants}

Retain the notation set up for a wonderful $G$-variety $X$ of rank $r$.
After Luna, we attach three combinatorial invariants to $X$ as follows.

The (unique) closed $G$-orbit $Y$ of $X$ yields the first invariant, a set of simple roots of $G$ denoted by $S^p_X$.
Let $P_X$ be the parabolic subgroup of $G$ containing $B$ and such that $Y\cong G/P_X$.
The set $S^p_X$ is precisely the set of simple roots of the Levi subgroup of $P_X$ containing $T$.

The second invariant is the set $\Sigma_X$ of \textsl{spherical roots of $X$} defined as the following set $\{\sigma_1,\ldots,\sigma_r\}$ of linearly independent characters of $T$.
Let $X_B$ be the complement in $X$ of the union of the colors of $X$; it is isomorphic to an affine space (\textsl{see}~\cite{Br89}).
Let $f_i\in k[X_B]$ be an equation (uniquely determined up to a non-zero scalar) of $X_B\cap D_i$.
Then $f_i$ is a $B$-eigenvector and the opposite of its $B$-weight is denoted by $\sigma_i$.
There is another characterization of the spherical roots of $X$, namely in terms of the $T$-weights of the normal space to the orbit $Y$ in $X$; \textsl{see}~\cite{L97} for details.

The third invariant is given by the set $\mathcal D_X$ and a collection of integers $a_{\sigma,D}$ indexed by $\Sigma_X\times \mathcal D_X$.
Let us now label the boundary divisors $D_i$ according to the spherical roots $\sigma_i$ of $X$.
For each $\sigma$ in $\Sigma_X$, we have
$$
[D_\sigma]=\sum_{D\in\mathcal D_X} a_{\sigma,D} [D]\quad \mbox{ in $\mathrm{Pic}(X)$}
$$
where the $a_{\sigma,D}$ are integers.
Equivalently, regarding the equations $f_i$ as $B$-weight vectors in the function field of $X$, we get
$$
(\sigma,0)=\sum_{D\in\mathcal D_X} a_{\sigma,D} (\omega_D,\chi_D).
$$

\subsection{Total coordinate ring}\label{Coxring}

Let $X$ be a wonderful $G$-variety and $\mathcal D=\mathcal D_X$ be its set of colors.

The following definition and results of this subsection are freely collected from Section 3 in~\cite{Br07}.

Set
$$
\tilde{G}=G\times \mathbb G_m^\mathcal D
$$
with
$\mathbb G_m^\mathcal D$ being the torus with character group $\mathbb Z^\mathcal D\cong\mathrm{Pic}(X)$.

Define the \textsl{total coordinate ring of $X$} as
$$
R(X)=\oplus_{(n_D)_D\in\mathbb Z^{\mathcal D}} H^0( X,\mathcal O_X(\sum_{D\in\mathcal D} n_D D)).
$$
This is a $\mathbb Z^\mathcal D$-graded finitely generated $k-$algebra. Further
$$
\tilde{X}:=\mathrm{Spec}\,R(X)
$$
is a factorial spherical $\tilde{G}$-variety.

\begin{proposition}[Proposition 3.11 in loc. cit.]\label{Brionfamily}
\smallbreak\noindent
{\rm(i)}\enspace
The canonical sections of the boundary divisors $D_\sigma$ ($\sigma\in\Sigma_X$) of $X$ from a regular sequence in $R(X)$ and generate freely the ring of invariants $R(X)^G$.
\smallbreak\noindent
{\rm(ii)}\enspace
The general fibers of the quotient morphism
$$
\pi:\tilde{X}\rightarrow \mathrm{Spec}(R(X)^G)
$$
are isomorphic to the spherical $G \times (H/H^\sharp)^\circ$-variety $\mathrm{Spec}\left(k[G/H^\sharp]\right)$.
\end{proposition}

Consider the action of $T$ on $\tilde{X}$ given by the homomorphism
$$
T\rightarrow \mathbb G_m^\mathcal D,\quad t\mapsto (\omega_D(t))_{D\in\mathcal D}.
$$
Recall the definition of the $f_D$'s with $D\in\mathcal D$ (after Definition~\ref{defwonderful}).

\begin{theorem}[Theorem 3.2.3 in loc. cit.]
There is an isomorphism of $(G\times T)$-algebras
$$
R(X)\cong \oplus_{\lambda,\mu} k[G/H^\sharp]_{(\lambda)} e^\mu
$$
where the sum runs over the dominant weights $\lambda\in\Lambda^+$ and characters $\mu$ of $T$
such that $\mu-\lambda$ is a linear combination of spherical roots of $X$ with non-negative coefficients;
the right-hand side is a subalgebra of $k[G\times T]$.

This isomorphism identifies the canonical section of any boundary divisor $D_\sigma$ with $e^\sigma$ and that of any color $D$ of $X$ with $f_D e^{\omega_D}$.
\end{theorem}

\section{Spherical systems}\label{Spherical systems}

The following objects were introduced by Luna in~\cite{Lu01} and have been inspired by Wasserman's classification of rank $2$ wonderful varieties.

\begin{definition}\label{sphericalroots}
\textsl{A spherical root of $G$} is the spherical root of a rank $1$ wonderful $G$-variety.
\end{definition}

Wonderful $G$-varieties of rank $1$ are by definition $2$-orbit varieties whose closed $G$-orbit is $1$-codimensional; they were classified by Akhiezer in~\cite{A} (\textsl{see} also~\cite{Br97}).
The spherical roots of $G$ are explicitly listed in~\cite{W}; \textsl{see} Appendix~\ref{listsphericalroots} for recollection.

\begin{definition}\label{def-sph-syst}
Let $S^p$ be a set of simple roots of $G$, $\Sigma$ a set of spherical roots of $G$ and $\mathbf A$ an abstract set equipped with a pairing
$c: \mathbf A\times \Sigma\rightarrow \mathbb Z$.
The triple $(S^p,\Sigma,\mathbf A)$ is called a \textsl{spherical system of $G$} if it satisfies the following axioms.
\smallbreak\noindent{\rm ($\mathbf A1$)}\enspace
$c(D_\alpha^\pm,\sigma)\leq 1$ for every $\alpha\in\Sigma\cap S$ and $\sigma\in\Sigma$; with equality only if $\sigma\in S$.
\smallbreak
\noindent{\rm ($\mathbf A2$)}\enspace
For any $\alpha\in\Sigma\cap S$, define $\mathbf A(\alpha)=\{D\in\mathbf A: c(D,\alpha)=1 \}$. 
Then $\mathbf A(\alpha)$ is of cardinality $2$ and $\mathbf A$ is the union of these sets.
\smallbreak\noindent
{\rm ($\mathbf A3$)}\enspace
$c(D_\alpha^-,\sigma)+c(D_\alpha^+,\sigma)=(\alpha^\vee,\sigma)$ for every $\alpha\in S\cap\Sigma$ and $\sigma\in\Sigma$ along with $\mathbf A(\alpha)=\{D_\alpha^\pm\}$.
\smallbreak\noindent
{\rm($\Sigma 1$)}\enspace $(\alpha^\vee,\sigma)\in2\mathbb Z_{\leq 0}$ for every $\sigma\in\Sigma\setminus\{2\alpha\}$ and $\alpha\in S\cap\frac{1}{2}\Sigma$.
\smallbreak
\noindent
{\rm($\Sigma 2$)}\enspace $(\alpha^\vee,\sigma)=(\beta^\vee,\sigma)$ for every $\sigma\in\Sigma$ and $\alpha,\beta\in S$ orthogonal and such that $\alpha+\beta\in\Sigma$.
\smallbreak\noindent
{\rm($S$)}\enspace For each $\sigma\in\Sigma$, the data $S^p$ and $\sigma$ are those of a rank $1$ wonderful $G$-variety.
\smallbreak\noindent
\smallbreak\end{definition}

\begin{definition}\label{sphericallyclosed}
A spherical system is \textsl{spherically closed} if for each of its spherical roots $\sigma$,  Axiom $(S)$ holds for $S^p$ and $2\sigma$ only if $\sigma\in S$.
\end{definition}

\begin{remark}
Axiom $(S)$ hence also the property of being spherically closed can be formulated in a purely combinatorially way; \textsl{see} Appendix~\ref{listsphericalroots} for details.
\end{remark}

\bigbreak
 
\subsection{Set of weights associated to a spherical system}~\label{colors}
The purpose of this subsection is to attach to any spherical system $\mathscr{S}=(S^p,\Sigma,\mathbf A)$ of $G$,
a set of linearly independent characters $(\omega_D,\chi_D)$ of $T\times C$ for some group $C$.

\subsubsection{}
These characters are indexed by a finite set $\Delta$, \textsl{the set of colors of $\mathscr{S}$}.
The set $\Delta$ is defined as follows (\textsl{see}~\cite{L97}).
Set
$$
S^a=S\cap \frac{1}{2}\Sigma\quad \mbox{ and }\quad S^b=S\setminus\left(S^p\cup (S\cap \Sigma)\cup S^a\right).
$$
The sets $S^p$, $(S\cap \Sigma)$ and $S^a$ are pairwise disjoint thanks to the axioms $(S)$ and $(\Sigma1)$.

If $\alpha$ and $\beta$ are orthogonal simple roots whose sum is an element of $\Sigma$, write $\alpha\sim\beta$.
Define now
$$
\Delta:=\mathbf A \sqcup S^{a}\cup S^b/\sim .
$$
In the remainder, we shall denote the elements $\alpha$ of $\Delta\setminus\mathbf A$ rather by $D_\alpha$.

\subsubsection{}
Let $\omega_\alpha$ denote the fundamental weight associated to the simple root $\alpha$.

Given $D\in\Delta$, we define (after \cite{F})
$$
\omega_{D}=\left\{
\begin{array}{ll}
\sum_{\alpha:D\in\mathbf A(\alpha)}\omega_{\alpha} &  \mbox{ if $D\in\mathbf A$}\\
2\omega_\alpha   &  \mbox{ if $D=D_\alpha$ with $\alpha\in S^a$}\\
\sum_{\alpha:D_\alpha=D}\omega_\alpha &\mbox{ otherwise}
\end{array}\right ..
$$
Note that these weights may not be pairwise distinct: in case $\alpha\in S\cap\Sigma$, the weight $\omega_\alpha$ may occur
twice, as shown right below -- but not more (since $\mathbf A(\alpha)$ is of cardinality $2$).

\begin{example}
The variety $X=\mathbb P^1\times\mathbb P^1$ equipped with the diagonal action of $SL_2$ is wonderful of rank $1$.
Its spherical root is the simple root of $SL_2$ and the set $S^p_X$ is empty.
This yields naturally a spherical system with $\mathbf A$ of cardinality $2$.
The set $\Delta$ of colors thus equals $\mathbf A$ and the associated weights $\omega_D$ are all equal to the fundamental weight $\omega_\alpha$.
\end{example}

\subsubsection{}
We now introduce some additional characters $\chi_D$ indexed by $\Delta$ (\textsl{see} also ~\cite{L97} and Lemma 3.2.1 with its proof in~\cite{Br07}).

Given $D\in\Delta$ and $\sigma\in\Sigma$, let us define (after~\cite{L97})
$$
a_{\sigma,D}=\left\{
\begin{array}{rl}
c(D^+_\alpha,\sigma) & \mbox{ if $D=D^+_\alpha$} \\
c(D^-_\alpha,\sigma) & \mbox{ if $D=D^-_\alpha$}\\
\frac{1}{2}(\sigma,\alpha^\vee) & \mbox{ if $D=D_\alpha$ with $\alpha\in S^a$}\\
(\sigma,\alpha^\vee)& \mbox{ for the remaining $D=D_\alpha$}
\end{array}\right. .
$$

The spherical roots in $\Sigma$ are linearly independent characters of $T$; \textsl{see} Lemma~\ref{lindependencesphroots} for a proof.
Let $\mathbb G_m^r$ be the torus whose character group is spanned freely by the set $\Sigma$ and let
$\mathbb G_m^\Delta$ be the torus with character group $\mathbb Z^\Delta$.
Consider the morphism
$$
\varphi:\mathbb G_m^\Delta\rightarrow \mathbb G_m^r: (t_D)_{D\in\Delta}\mapsto \left( \prod_{D\in\Delta} t_D^{a_{\sigma,D}}\right)_{\sigma\in\Sigma}.
$$
Let $C$ be its kernel; it is a diagonalizable group.

Define the character $\chi_D$ as the restriction to $C$ of the $D$-component character
$$
\varepsilon_D:(t_D)_{D\in\Delta}\mapsto t_D .
$$

\begin{lemma}\label{enlargeddefinitionofweights}
The characters $(\omega_D,\chi_D)$ of $T\times C$ are linearly independent.
Further they satisfy the following equalities
\begin{equation}\label{sphericalversusweights}
(\sigma,0)=\sum_{D\in\Delta} a_{\sigma,D} (\omega_D,\chi_D)\quad\mbox {for all $\sigma\in\Sigma$}.
\end{equation}
\end{lemma}

\begin{proof}
By definition of the weights $\chi_D$, we have $\sum_{D\in\Delta} a_{\sigma,D} \chi_D=0$ for every $\sigma\in\Sigma$.
Given $\alpha\in S$, we have $(\omega_D,\alpha^\vee)=0, 1$ or $2$. Further,
$\sum_{D\in\Delta} a_{\sigma,D} (\omega_D,\alpha^\vee)$ consists of at most two non-trivial terms.
Specifically, we have for all $\sigma\in\Sigma$
$$
\sum_{D\in\Delta} a_{\sigma,D} (\omega_D,\alpha^\vee)=\left\{
\begin{array}{rl}
0 & \mbox{ if $\alpha\in S^p$} \\
a_{\sigma,D_{\alpha}^+}+a_{\sigma,D_\alpha^-} & \mbox{ if $\alpha\in S\cap\Sigma$}\\
2a_{\sigma,D_\alpha} & \mbox{ if $\alpha\in S^{2a}$}\\
a_{\sigma,D_\alpha}& \mbox{ otherwise}
\end{array}\right. .
$$
This together with the very definition of the scalars $a_{\sigma,D}$ implies the equality (\ref{sphericalversusweights}).

Consider now the morphism $\psi:T\times C\rightarrow \mathbb G_m^\Delta$ defined naturally by the characters $(\omega_D,\chi_D)$.
The composition $\varphi\circ\psi:T\times C\rightarrow\mathbb G_m^r$ is then an epimorphism thanks to (\ref{sphericalversusweights}) along with the linear independence of the spherical roots in $\Sigma$ (Lemma~\ref{lindependencesphroots}).
The morphism $\psi$ is thus in turn an epimorphism.
The linear independence of the characters $(\omega_D,\chi_D)$ follows.
\end{proof}

The set of characters $(\omega_D,\chi_D)$ will be referred in the remainder as \textit{the set of weights associated to $\mathscr{S}$}.

\subsection{Relations with wonderful varieties}

Recall the notation set up in Section~\ref{wondervar}.

\subsubsection{}\label{colorsvswonderful}

For a given wonderful $G$-variety $X$, the third invariant previously associated to $X$ may be refined.
Instead of the whole set $\mathcal D_X$ of colors, we take the following subset $\mathbf A_X$ of it.
Let $\alpha\in S$ and $P_\alpha$ be the corresponding minimal parabolic subgroup of $G$ containing $B$.
Let $\mathcal D_X(\alpha)$ denote the set of colors $D$ of $X$ such that $P_\alpha. D\neq D$.
The set $\mathbf A_X$ is defined as the union of the $\mathcal D_X(\alpha)$'s where $\alpha\in\Sigma_X$.
As third invariant for $X$, we take the set $\mathbf A_X$ and the pairing on $\mathbf A_X\times \Sigma_X$ defined by the 
integers $a_{\sigma,D}$ indexed by $\Sigma_X\times\mathbf A_X$.

\begin{theorem}[\cite{Lu01}]
Suppose $G$ is of adjoint type, \emph{i.e.} the center of $G$ is trivial.
The triple $(S^p_X,\Sigma_X,\mathbf A_X)$ associated to a wonderful $G$-variety $X$ is a spherical system of $G$.
\end{theorem}

\begin{conjecture}[\cite{L97}]
Suppose $G$ is of adjoint type.
The map which assigns to any wonderful $G$-variety $X$ the triple $(S^p_X,\Sigma_X,\mathbf A_X)$ defines a bijection between the
set of isomorphism classes of wonderful $G$-varieties and the set of spherical systems of $G$.
\end{conjecture}

We will prove this conjecture in Section~\ref{proofofconjecture}.

\subsubsection{}
Thanks to the results obtained in Section 3 of~\cite{L97} as well as Lemma 3.2.1 (and its proof) in~\cite{Br07}, we have:

\begin{proposition}\label{colors-sphericalsystems-wful}
Let $X$ be a wonderful $G$-variety.
The set of $(B\times H)$-weights of the equations $f_D$ ($D\in\mathcal D_X$) is the set of weights associated to the spherical system $\mathscr{S}_X$ of $X$.
Further, the diagonalizable group $C$ attached to $\mathscr{S}_X$ is the group $H/H^\sharp$.
\end{proposition}

Finally, let us mention that the set of colors of a wonderful variety $X$ coincides with that of its spherical system.

\section{Invariant Hilbert schemes}

\subsection{Definition}

The definitions and results stated in this section are taken from ~\cite{AB} except that they are formulated in a more general setting in loc. cit..
One may consult also the survey~\cite{Br10}.

Let $\lambda_1$, \ldots, $\lambda_s$ be linearly independent weights in $\Lambda^+$.
Denote by $\underline\lambda$ the corresponding s-tuple and by $\Gamma$ the submonoid of $\Lambda^+$ spanned by the $\lambda_i$'s.
Set
$$
V:=V(\lambda_1^*)\oplus\ldots\oplus V(\lambda_s^*).
$$

\begin{definition}
Given a scheme $S$, \textsl{a family $\mathcal X$ of closed $G$-subschemes of $V$ over $S$ of type $\Gamma$}
is a closed $G$-subscheme of $V\times S$ such that
\begin{enumerate}
	\item the projection $\pi:\mathcal X\rightarrow S$ is $G$-invariant;
	\item the sheaf $\mathcal F_\lambda:=(\pi_*\mathcal O_{\mathcal X})_\lambda^U$ of $\mathcal O_S$-modules is invertible for every $\lambda\in\Gamma$ and $0$ otherwise.
\end{enumerate}
\end{definition}

With the preceding notation, the sheaf $\pi_*\mathcal O_{\mathcal X}$ is isomorphic (as an $(\mathcal O_S\times G)$-module) to $\oplus_{\lambda\in\Gamma} \mathcal F_\lambda\otimes V(\lambda)$; \textsl{see} Lemma 1.2. in loc. cit.

\begin{remark}
Since no confusion can arise, $S$ denotes throughout this section a scheme and not the set of simple roots as stated previously.
\end{remark}

\begin{theorem}[Theorem 1.7 in~\cite{AB}]
The functor
which assigns to any scheme $S$ the set of families
of closed $G$-subschemes of $V$ over $S$ of type $\Gamma$ is representable by a quasi-projective scheme, \textsl{the invariant Hilbert scheme} $\mathrm{Hilb}^G_{\Gamma}(V)$.

\end{theorem}

In particular, $\mathrm{Hilb}^G_{\Gamma}(V)$ contains as closed point the $G$-variety $X_0=X_0(\underline\lambda)$ given by the $G$-orbit closure within $V$ of
$$
v_{\underline\lambda^*}=v_{\lambda_1^*}+\ldots+v_{\lambda_s^*}
$$
where $v_{\lambda_i^*}$ stands for a highest weight vector in $V$ of weight $\lambda_i^*$.
Note that $X_0$ is a spherical $G$-variety thanks to the criterion recalled in Section~\ref{affinesphericitycriterion}.
More generally, any closed point of $\mathrm{Hilb}^G_{\Gamma}(V)$  is a spherical $G$-variety.

A subvariety of $V$ is called \textsl{non-degenerate} if its projection onto $V(\lambda_i^*)$ is non-trivial for every $i=1,\ldots , s$.

\begin{theorem}[\cite{AB}]\label{AB-representability}
The non-degenerate irreducible subvarieties of $V$ which can be regarded as closed points of $\mathrm{Hilb}^G_{\Gamma}(V)$ are parameterized by a connected affine open subscheme $\mathrm{Hilb}^G_{\underline\lambda}$ of $\mathrm{Hilb}^G_{\Gamma}(V)$.
\end{theorem}

\begin{remark}\label{connectionwithmoduli}
Theorem~\ref{AB-representability} gathers several results of ~\cite{AB} together.
Specifically, since the weights $\lambda_i$ are assumed to be linearly independent, $\mathrm{Hilb}^G_{\underline\lambda}$ can be identified to the so-called \textsl{moduli scheme $M_\Gamma$ of multiplicity-free varieties with weight monoid $\Gamma$}; \textsl{see} Corollary 1.17 in \cite{AB} or Example 4.8 in~\cite{Br10}.
We do not recall the intrinsic definition of $M_\Gamma$ since it will not be used explicitly. 
Further by Theorem 2.7 in \cite{AB}, $M_\Gamma$ is affine and connected. 
\end{remark}

\subsection{Toric action}\label{Toricaction}
Let $Z(G)$ be the center of $G$.
We recall briefly how the action of the adjoint torus $T_\ad:=T/Z(G)$ on $\mathrm{Hilb}^G_{\underline\lambda}$ is defined; \textsl{see} Section 2.1 in~\cite{AB} for details.

Consider a family $\pi:\mathcal X\rightarrow  S$ of closed $G$-subschemes of $V$ of type $\Gamma$.
Let $Z(G)$ act on $\mathcal X\times T$ by $z.(x,t)=(z.x,z^{-1}t)$.
Then
$$
\tilde{\mathcal X}=(\mathcal X\times T)/Z(G)
$$
is a scheme equipped with an action of $G$.
The morphism $\pi\times\mathrm{id} :\mathcal X\times T\rightarrow S\times T$ descends to a morphism
$$
\tilde{\pi}:\tilde{\mathcal X}\rightarrow (S\times T)/Z(G)=S\times T_\ad .
$$
Moreover, we have an isomorphism of $G-\mathcal O_S\otimes k[T_\ad]$-modules
$$
\tilde{\pi}_*\mathcal O_{\tilde{\mathcal X}}\simeq\oplus_{\lambda\in\Lambda^+}(\pi_*\mathcal O_{\mathcal X})_{(\lambda)} e^\lambda\otimes_k k[T_\ad].
$$

Consider the action of $G$ on $V(\lambda)\times T_\ad$ given by: $g.(v,s)=(gv,s)$ and the action of $T$ on $V(\lambda)\times T_\ad$ via

\begin{equation}\label{normalisedaction}
t.(v,s)=(w_0(\lambda)(t)t^{-1}v,ts).
\end{equation}

Set
$$
\mathbf{G}:=(G\times T)/Z(G).
$$
Let $\mathcal X=V(\lambda)$. With the notation set up above, the $\mathbf{G}$-schemes $\widetilde{\mathcal X}$ and
$V(\lambda)\times T_\ad$ turn out to be  isomorphic and in turn so are $\tilde V$ and $V\times T_\ad$.

The scheme$(\tilde{\mathcal X},\tilde{\pi})$ thus defines a family of closed $G$-subschemes of $V$ of type $\Gamma$ over $S\times T_\ad$.
Applying this construction to the universal family, 
one obtains a morphism of schemes
$$
a:T_\ad\times \mathrm{Hilb}^G_{\Gamma}(V)\rightarrow \mathrm{Hilb}^G_{\Gamma}(V).
$$

\begin{theorem}[Proposition 2.1/ Theorem 2.7 in~\cite{AB}]\label{ABtoricaction}
The morphism $a$ defines an action of $T_\ad$ on $\mathrm{Hilb}^G_{\Gamma}(V)$ and this action stabilizes $\mathrm{Hilb}^{G}_{\underline\lambda}$.
Furthermore under this action, $\mathrm{Hilb}^{G}_{\underline\lambda}$ contracts to $X_0$.
In particular, $X_0$ is the unique fixed point of $T_\ad$ in $\Hilb^G_{\underline\lambda}$.
\end{theorem}

\begin{remark}
Similarly as Theorem~\ref{AB-representability}, the two last assertions of Theorem~\ref{ABtoricaction} involves the identification of $\mathrm{Hilb}^{G}_{\underline\lambda}$ with $M_\Gamma$.
\end{remark}


\subsection{Tangent space}
Let $G_{v_{\underline\lambda^*}}$ be the isotropy group in $G$ of $v_{\underline\lambda^*}$
and let $\mathfrak g$ be the Lie algebra of $G$.
Note that $G_{v_{\underline\lambda^*}}$ stabilizes $\mathfrak g.v_{\underline\lambda^*}$ whence an action of $G_{v_{\underline\lambda^*}}$ on the normal space 
$V/\mathfrak g.v_{\underline\lambda^*}$.

Consider the action of the adjoint torus $T_\ad$ on the scheme $\mathrm{Hilb}_{\underline \lambda}$ recalled in Subsection~\ref{Toricaction}.
Since $X_0$ is a $T_\ad$-fixed point for this action, the tangent space $T_{X_0}\mathrm{Hilb^G_{\underline\lambda}}$ of $\mathrm{Hilb^G_{\underline\lambda}}$ at $X_0$ carries
obviously a $T_\ad$-module structure on $T_{X_0}\mathrm{Hilb}_{\underline \lambda}$.
In the following, we refer to this module structure.

Moreover, we consider the \emph{normalised action} of the adjoint torus $T_\ad$ on $V$, that is the action naturally induced by that defined in~(\ref{normalisedaction}).
This action commutes with the action of $G$ on $V$ and induces in particular an action of $T_\ad$ on $(V/{\mathfrak g}.v_{\underline\lambda^*})^{G_{v_{\underline\lambda^*}}}$.

\begin{proposition}[Proposition 1.13 and 1.15 in~\cite{AB}]\label{tgtAB}
Let $\mathcal N_{X_0/V}$ denote the normal sheaf of $X_0$ in $V$.
\smallbreak\noindent
{\rm(i)}	\enspace The tangent space $T_{X_0}\mathrm{Hilb}^G_{\underline\lambda}$  is canonically isomorphic to the invariant space $H^0(X_0,\mathcal N_{X_0/V})^G$. \smallbreak\noindent
{\rm(ii)}	\enspace The restriction map $H^0(X_0,\mathcal N_{X_0/V})^G\rightarrow H^0(G.v_{\underline\lambda^*},\mathcal N_{X_0/V})^G$
yields an injection of $T_\ad$-modules
$$
T_{X_0}\mathrm{Hilb}^G_{\underline\lambda}\hookrightarrow\left(V/\mathfrak gv_{\underline\lambda^*}\right)^{G_{v_{\underline\lambda^*}}}.
$$
\smallbreak\noindent
{\rm(iii)}	\enspace If the boundary $X_0\setminus G.v_{\underline\lambda^*}$ is of codimension at least $2$ then
the above injection is an isomorphism of $T_\ad$-modules.
\end{proposition}

In the next proposition, we adopt the following notation.
Given a $T_\ad$-weight vector $[v_\gamma]$ of $\left(V/\mathfrak g.v_{\underline\lambda^*}\right)^{G_{v_{\underline\lambda^*}}}$ of weight $\gamma$, let $s_\gamma$ denote the corresponding section in $H^0(G.v_{\underline\lambda^*},\mathcal N_{X_0/V})^G$ that is, 
$$
s_\gamma(v_{\underline \lambda^*})=[v_\gamma].
$$
Let $\rho_1,\ldots,\rho_s$ form the basis dual to $\lambda_1^*,\ldots, \lambda_s^*$.
For each $\lambda_i^*$,  we write $\underline{\hat\lambda_i^*}$ for the $(s-1)$-tuple given by all the $\lambda_k^*$ but $\lambda_i^*$.
The set of simple roots $\alpha$ which are orthogonal to every $\lambda_k^*$ without exception (resp. but $\lambda_i^*$) is denoted by $\underline{ \lambda^*}^\perp$
(resp. $\underline{\hat \lambda_i^*}^\perp$).

\begin{proposition}\label{extension-section}
Let $[v_\gamma]$ be a $T_\ad$-weight vector of $\left(V/\mathfrak g.v_{\underline\lambda^*}\right)^{G_{v_{\underline\lambda^*}}}$ of weight $\gamma$
with $v_\gamma\in \oplus_k V(\lambda_k^*)_{\lambda_k^*-\gamma}$.
\begin{enumerate}[label=\textup{(\arabic*)},ref=\textup{\arabic*}]
\item
If $\rho_i(\gamma)\leq 0$ for all $i=1,\ldots, s$ such that $\underline{\lambda^*}^\perp=\underline{\hat\lambda_i^*}^\perp$ then the section $s_\gamma$ extends to $X_0$.
\item 	
Let $\rho_i(\gamma)>0$ for some $i=1,\ldots, s$ such that $\underline{\lambda^*}^\perp=\underline{\hat\lambda_i^*}^\perp$.
If further the projection of $v_\gamma$ onto $\oplus_{k\neq i}V(\lambda_k^*)$ does not belong to $\mathfrak g.v_{\underline {\hat\lambda_i^*}}$ 
then $s_\gamma$ does not extend to $X_0$.
\end{enumerate}

\end{proposition}

\begin{proof}
This proof has been inspired by the content of Section 3 in~\cite{PVS}.
 
First, one should note that the section $s_\gamma$ extends to $X_0$ if and only if $s_\gamma$ extends to the union of all $1$-codimensional $G$-orbits of $X_0$.
Further, the $1$-codimensional $G$-orbits of $X_0$ are given by the $G$-orbit closures within $V$ of the vectors $v_{\underline{\hat\lambda_i}}$ with $\lambda_i$   satisfying the condition $\underline{\lambda^*}^\perp=\underline{\hat\lambda_i^*}^\perp$. If necessary, see e.g. loc. cit. for details.
Consequently, to define an extension of $s_\gamma$ to $X_0$, it suffices to define $s_\gamma(v_{\underline{\hat{\lambda_i^*}}})$ for all such $\lambda_i$.

Moreover, since $v_{\underline{\hat{\lambda_i^*}}}$ belongs to the $T$-orbit closure of $v_{\underline \lambda^*}$, to prove the proposition,
we shall consider
$\lim_{n\rightarrow\infty} s(t_n v_{\underline{ \lambda^*}})$ for any sequence of elements $t_n$ of $T$ such that 
$$
 \lim_{n\rightarrow\infty} t_n v_{\underline \lambda^*}=v_{\underline{\hat{\lambda_i^*}}}\quad
 \mbox{ with}\quad\underline{\lambda^*}^\perp=\underline{\hat\lambda_i^*}^\perp.
$$
In particular, $\lim_{n\rightarrow\infty}\lambda_i(t_n)=0$.

Let $t_n\in T$. We have: 
$s_\gamma(t_n v_{\underline{ \lambda^*}})=t_n s_\gamma(v_{\underline{ \lambda^*}})$ since $s_\gamma$ is $G$-invariant.
We thus have: 
$$
s_\gamma(t_n v_{\underline \lambda^*})= \big [\sum_k (\lambda_k^*-\gamma)(t_n)v_\gamma^{\lambda_k^*}\big ]\in V/T_{t_n v_{\underline {\lambda^*}}}X_0.
$$
Recall that $\gamma\in\oplus_k\mathbb Z \lambda_k^*$.

Let $(t_n)_n$ be any sequence  as above. 
If $\rho_i(\gamma)\leq 0$ then $\lim_{n\rightarrow\infty} s_\gamma(t_n v_{\underline \lambda^*})$ exists in $V$ and it is independent of the choice of $(t_n)_n$; the first assertion thus follows.

Let $\rho_i(\gamma)> 0$. By assumption, there exists $k\neq i$ such that $v_\gamma^{\lambda_k^*}\neq 0$.
It follows that
$\lim_{n\rightarrow\infty}(\lambda_k^*-\gamma)(t_n)v_\gamma^{\lambda_k^*}$ does not exist in $V$ and neither does $\lim_{n\rightarrow\infty} s_\gamma(t_n v_{\underline \lambda^*})$.
This concludes the proof of the second assertion.
\end{proof}

\subsection{Invariant infinitesimal deformations}

For local studies purposes (e.g. smoothness of $\mathrm{Hilb}^G_{\underline\lambda}$), we shall need to consider the
functor of \textsl{invariant infinitesimal deformations of $X_0$}
$$
{\mathcal Def}^G_{X_0}: \mathcal A\rightarrow (\mathrm{Sets})
$$
where $\mathcal A$ denotes the category of local Artinian $k$-algebras.
Given $A\in\mathcal A$, we define ${\mathcal Def}^G_{X_0}(A)$ as the set of Cartesian diagrams
$$
\begin{CD}
X_0@>>>  \mathcal X\\
@VVV @VVV\\
\mathrm{Spec}(k)@>>> \mathrm{Spec}(A)
\end{CD}
$$
with $\mathcal X\rightarrow \mathrm{Spec(A)}$ being a family of closed $G$-subschemes of $V$ of type $\Gamma$.

By Theorem~\ref{AB-representability} (\emph{see} also e.g. Section 2.2 in~\cite{Se}), the functor ${\mathcal Def}^G_{X_0}$ is representable by
the completion $\hat{\mathcal O}_{\mathrm{Hilb}^G_{\underline\lambda},[X_0]}$.

\subsection{Obstruction space}

First, let us recall (\emph{see} for instance \cite{Se}) the definition and the main properties of the obstruction space of a covariant functor $\mathcal F:\mathcal A\rightarrow (\mathrm{Sets})$.

Given $A\in\mathcal A$, let $\mathrm{Ex}(A,k)$ denote the $A$-module of isomorphism classes of $k$-extensions of $A$ by $k$.
An element $[(\tilde A,\varphi)]$ of $\mathrm{Ex}(A,k)$ is thus represented by an exact sequence
\begin{equation}\label{extension}
\begin{CD}
(\tilde{A},\varphi): 0@>>> k\varepsilon@>>> \tilde A @>\varphi>> A@>>> 0\quad\mbox{ with $\varepsilon^2=0$}.
\end{CD}
\end{equation}

\begin{definition}
A $k$-vector space $v(\mathcal F)$ is called an \textsl{obstruction space for the functor $\mathcal F$} if
for every object $A$ of $\mathcal A$ and every $\xi\in \mathcal F(A)$,
there exists a $k$-linear map
$$
\xi_v:\mathrm{Ex}(A,k)\rightarrow v(\mathcal F)
$$
with the following property:
$\ker(\xi_v)$ consists of the isomorphism classes of extensions $(\tilde A,\varphi)$
such that
$\xi\in\mathrm{Im} (\mathcal F(\tilde A)\stackrel{\mathcal F(\varphi)}{\rightarrow} \mathcal F(A))$.
\end{definition}

If the functor $\mathcal F$ has a trivial obstruction space then the functor $\mathcal F$ is smooth, i.e. $\mathcal F(p): \mathcal F(B)\rightarrow \mathcal F(A)$ is surjective for every surjection $p:B\rightarrow A$ of $\mathcal A$.

\subsubsection{}
Let $B_0$ be the coordinate ring of $X_0\subset V$.
Following Section 3.1.2 in~loc. cit, we recall the definition of the second cotangent module $T^2_{B_0}$ of $B_0$.

Take a presentation of the ideal $I\subset \mathrm{Sym}(V^*)$ of $X_0$ as $\mathrm {Sym}(V^*)-G$-module
$$
\begin{CD}
0 @>>>R @>>> F @>\phi>>I @ >>> 0
\end{CD}
$$
where $F$ is a finitely generated free $\mathrm {Sym}(V^*)$-module.

Consider the module $K\subset R$ of trivial relations: $K$ is generated by the relations
$\phi(e_i)e_j-\phi(e_j)e_i$ with $e_1,\ldots,e_n$ being a basis of the $\mathrm {Sym}(V^*)$-module $F$. We thus get the exact sequence of $B_0$-modules

$$
R/K\rightarrow F\otimes B_0\rightarrow I/I^2\rightarrow 0.
$$

Apply $\mathrm{Hom}_{B_0}(-,B_0)$ to the last exact sequence then
the second cotangent module $T^2_{B_0}$ of $B_0$ is defined by the exact sequence
$$
\mathrm{Hom}_{B_0}(I/I^2,B_0)\rightarrow \mathrm{Hom}_{B_0}(F\otimes B_0,B_0) \rightarrow\mathrm{Hom}_{B_0}(R/K,B_0)\rightarrow T^2_{B_0}\rightarrow 0.
$$
The second cotangent module of $B_0$ is independent of the presentation of $I$. Moreover, it is supported on the singular locus of $X_0$.

\subsubsection{}
As proved in~\cite{Se}, $T^2_{B_0}$ is an obstruction space for the functor of deformations of $X_0$.
From this, we shall derive an invariant version of this statement and give a representation theoretic characterization of $(T^2_{B_0})^G$.

For any weights $\lambda\neq\nu$ in $\Lambda^+$, denote
$$
V(\lambda)\cdot V(\nu)= V(\lambda)\otimes V(\nu)
\quad\mbox{ and} \quad
V(\lambda)\cdot V(\lambda)=\mathop{S^2} V(\lambda).
$$
For any $1\leq i,j\leq s$, let
$$
m_{i,j}: V(\lambda_i^*)\cdot V(\lambda_j^*)\rightarrow V(\lambda_i^*+\lambda_j^*)
$$
be the natural projection onto the Cartan component $V(\lambda_i^*+\lambda_j^*)$ in $V(\lambda_i^*)\cdot V(\lambda_j^*)$ and
$$
K_{i,j}=\ker (m_{ij})\simeq V(\lambda_i^*)\cdot V(\lambda_j^*)/V(\lambda_i^*+\lambda_j^*).
$$
Then the duals $K_{i,j}^*$ generate the ideal $I$ of $X_0$; \textsl{see}~\cite{KR} for instance.

In the following, we take the presentation of $I/I^2$ as $B_0-G$-modules given by
\begin{equation}\label{presentation}
R/K\rightarrow \oplus_{i,j}B_0\otimes K_{i,j}^*\rightarrow I/I^2\rightarrow 0;
\end{equation}
\textsl{see}~loc. cit..

Let $v^i$ denote the projection of $v\in V$ onto $V(\lambda_i^*)$.
The assignment
$$
V\rightarrow \oplus_{i,j} V(\lambda_i^*)\cdot V(\lambda_j^*),\quad
v\mapsto \sum_{i,j} v^i\cdot v_{\lambda_j^*}
$$
yields obviously a map
$$
f: V/\mathfrak g.v_{\underline\lambda^*}\rightarrow \oplus_{i,j} K_{i,j}.
$$

\begin{proposition}\label{obstruction-cotgt}
\smallbreak\noindent
{\rm(i)}\enspace
The $G$-invariant space $(T^2_{B_0})^G$ of the second cotangent module of $B_0$ is an obstruction space for ${\mathcal Def}^G_{X_0}$.
\smallbreak\noindent
{\rm(ii)}\enspace
There is an injection of $(T^2_{B_0})^G$ into the kernel of the map
$$
H^1(f):H^1(G_{v_{\underline\lambda^*}},V/\mathfrak g.v_{\underline\lambda^*})\rightarrow
\oplus_{i,j} H^1(G_{v_{\underline\lambda^*}},K_{i,j})
$$
induced by the map $f$.
\smallbreak\noindent
{\rm(iii)}\enspace
If the boundary $X_0\setminus G.v_{\underline\lambda^*}$ is of codimension at least $2$ in $X_0$ then the aforementioned injection is an isomorphism.
\end{proposition}

\begin{proof}
The proof of the first assertion is essentially the "invariant version" of the classical one.

Let $A$ be a local Artinian $k$-algebra and $\xi\in{\mathcal Def}^G_{X_0}( A)$.
By Proposition 3.1.12 in \cite{Se}, there exists a $k$-linear map $\hat{\xi}:\mathrm{Ex}(A,k)\rightarrow T^2_{B_0}$
satisfying the property of the definition of an obstruction space with $\mathcal F$ being the functor of deformations of $X_0$.
Accordingly, we only need to prove that the image of $\hat{\xi}$ is $G$-invariant.
This follows from the very definition of $\hat{\xi}$ that we recall now.

Let $I=(f_1,\ldots,f_n)\in \mathrm{Sym}(V^*)$ denote the ideal of $X_0$ then $B_0=\mathrm{Sym}(V^*)/I$.
By definition, the given element $\xi$ can be regarded as a $G$-stable ideal $J$ of $\mathrm{Sym}(V^*)\otimes_k A$
generated by elements $F_1,\ldots, F_n$ such that $f_i-F_i\in (m_A \mathrm{Sym}(V^*))$ where $m_A$ stands for the maximal ideal of $A$.
Thanks to the flatness of $\mathrm{Sym}(V^*)\otimes_k A/J$ over $A$, for every relation $\underline r=(r_1,\ldots, r_n)\in R$,
there exist $R_1,\ldots, R_n$ in $\mathrm{Sym}(V^*)\otimes_k A$ such that $r_i=R_i$ modulo $(m_A \mathrm{Sym}(V^*))$ and $\sum_i R_i F_i=0$.

Take $[(\tilde{A},\varphi)]\in \mathrm{Ex}(A,k)$ and an exact sequence (\ref{extension}) as a representative.

Let $\tilde{F}_i$ (resp. $\tilde{R}_i$) be a lifting of $F_i$ (resp. $R_i$) through $\varphi$ for $i=1,\ldots,n$.
Then $\sum_i \tilde{R}_i \tilde{F}_i$ belongs to $\mathrm{Sym}(V^*)\varepsilon$ and may be regarded as an element of $\mathrm{Sym}(V^*)$.
As shown in loc. cit., the assignment
$$
\underline r\mapsto \sum_i \tilde{R}_i \tilde{F}_i
$$
thus defines an element of $\mathrm{Hom}(R,B_0)$ and yields naturally an element of $T^2_{B_0}$: the element $\hat{\xi}$ under consideration.

(ii) Let $\mathcal N_{\underline\lambda}$ be the normal sheaf of $G.v_{\underline\lambda}$.
Note that $K_{i,j}\otimes_k\mathcal O_{X_0}$
(resp. $(R/K)^*\otimes_{B_0}\mathcal O_{X_0}$) is the dual sheaf $\mathrm{Hom}_{\mathcal O_{X_0}}(K_{i,j}^*\otimes_k \mathcal O_{X_0}, \mathcal O_{X_0})$ (resp. $\mathrm{Hom}_{\mathcal O_{X_0}}(R/K\otimes_{B_0} \mathcal O_{X_0}, \mathcal O_{X_0})$).

First note that since $X_0$ is normal (being spherical as observed in~Section 3.1), the sheaf $(R/K)^*\otimes_{B_0}\mathcal O_{X_0}$ is reflexive and in turn the restriction map
$H^0(X_0,(R/K)^*\otimes_{B_0}\mathcal O_{X_0})\rightarrow H^0(G.v_{\underline\lambda^*},(R/K)^*\otimes_{B_0}\mathcal O_{G.v_{\underline\lambda^*}})$
is injective.

From the presentation (\ref{presentation}) of $I/I^2$, we obtain the diagram:
$$
\begin{CD}
0                                                        @.               \\
@VVV                                                                               \\
H^0(X_0,(R/K)^*\otimes\mathcal O_{X_0})                     @>>>   T^2_{B_0} @>>> 0 \\
@VVV @VVV \\
H^0(G.v_{\underline\lambda^*},(R/K)^*\otimes\mathcal O_{G.v_{\underline\lambda^*}})@>>> H^1(G.v_{\underline\lambda^*}, \mathcal N_{\underline\lambda^*}) @>>> H^1(G.v_{\underline\lambda^*}, \oplus_{i,j} K_{i,j}\otimes\mathcal O_{G.v_{\underline\lambda^*}})
\end{CD}
$$
Note that the last row of the above diagram follows from the fact that the second cotangent module $T^2_{B_0}$ is supported on the singular locus of $X_0$, the latter being contained in $X_0\setminus G.v_{\underline\lambda^*}$.

The normal sheaf $\mathcal N_{\underline\lambda^*}$
being the $G$-linearized sheaf on $G/G_{v_{\underline\lambda^*}}$ associated to the $G_{v_{\underline\lambda^*}}$-module $V/\mathfrak g.v_{\underline\lambda^*}$, we have (\textsl{see} I.5 in~\cite{Jt}):
$$
H^1(G.v_{\underline\lambda^*},\mathcal N_{\underline\lambda^*})^G=H^1(G_{v_{\underline\lambda^*}},V/\mathfrak g.v_{\underline\lambda^*})
$$
and similarly
$$
H^1(G.v_{\underline\lambda^*},\oplus_{i,j} K_{i,j}\otimes\mathcal O_{G.v_{\underline\lambda^*}})^G=
\oplus_{i,j} H^1(G_{v_{\underline\lambda^*}},K_{i,j}).
$$
The second assertion thus follows.

(iii) Since the codimension $X_0\setminus G.v_{\underline\lambda^*}$ in $X_0$ is $\geq 2$, the vertical injection of
the above diagram becomes an isomorphism whence the last assertion of the proposition.

\end{proof}

\begin{corollary}\label{Schlessinger}
If $(T^2_{B_0})^G$ is trivial then $\mathrm{Hilb}_{\underline\lambda}^G$ is smooth.
\end{corollary}

\begin{proof}
As already recalled, if $(T^2_{B_0})^G$ is trivial then the functor of invariant deformations of $X_0$ is smooth;
the latter being represented by $\hat{\mathcal O}_{\mathrm{Hilb}^G_{\underline\lambda^*},[X_0]}$, $\mathrm{Hilb}_{\underline\lambda^*}^G$ is smooth at $X_0$.
Since $X_0$ is the unique closed point of $\mathrm{Hilb}^G_{\underline\lambda^*}$ fixed by $T_\ad$ (Theorem~\ref{ABtoricaction}), the assertion follows.
\end{proof}

\section{Geometrical construction of wonderful varieties}

Throughout this section, $\mathscr{S}=(S^p,\Sigma,\mathbf A)$ denotes a spherical system of a simply connected semisimple group $G$.

Let $\Delta$ be the set of colors of $\mathscr{S}$.
Recall the definition of the diagonalizable subgroup $C\subset \mathbb G_m^\Delta$ as well
as the set of characters $\lambda_D=(\omega_D,\chi_D)$ (with $D\in\Delta$) of $T\times C$ -- both canonically associated to $\mathscr{S}$.
Set
$$
\tilde{G}=G\times C^\circ\quad \mbox{ and }\quad V=\oplus_{D\in\Delta} V\left(\lambda_D\right)^*.
$$

Note that by the definition of the weights $\chi_D$, the action of $C$ on $V$ is the diagonal action given by
$$
t. v_D= t_D v_D\quad\mbox{ for $v_D\in V(\lambda_D^*)$ and $t=(t_D)_D\in C$}.
$$

\subsection{Invariant Hilbert scheme attached to a spherical system}
As proved in Lemma~\ref{enlargeddefinitionofweights}, the weights $\lambda_D$ are linearly independent hence we may consider the relative invariant Hilbert scheme $\mathrm{Hilb}^{\tilde{G}}_{\underline\lambda}$ with 
$$
\underline\lambda=(\lambda_D)_{D\in\Delta}.
$$
In order to keep track of the datum $\mathscr{S}$,
let us denote the scheme $\mathrm{Hilb}^{\tilde{G}}_{\underline\lambda}$ rather by $\mathrm{Hilb(\mathscr{S})}$.
Further, put
$$
v_{\Delta^*}= v_{\underline\lambda^*}=\sum_{D\in \Delta}v_{\lambda_D^*}.
$$

\begin{remark}\label{saturatedsetting}
When the third datum of a spherical system is empty (\textsl{i.e.} $S\cap\Sigma=\emptyset$),
the invariant Hilbert scheme associated to the group $G$ itself and to $V$ as a $G$-module
falls in the case studied in~\cite{Js,BC1}.
Further, it maps naturally to $\Hilb (\mathscr{S})$.
\end{remark}

\begin{theorem}\label{alaKostant}
The $T_\ad$-module $(V/\tilde{\mathfrak g}.v_{\Delta^*})^{\tilde G_{v_{\Delta^*}}}$ is multiplicity-free.
Its weights are the opposites of the elements of some set $\Sigma(\Delta)$ such that
$$
\Sigma\subset\Sigma(\Delta)\subset\Sigma\cup \left\{\alpha+\alpha': \mbox{$\alpha,\alpha'$ adjacent (distinct) simple roots in $\Sigma$}\right\}.
$$
\end{theorem}

\begin{proof}
This stems from Proposition~\ref{descriptionofweights}.
\end{proof}

\begin{corollary}\label{tangentspace}
The tangent space $T_{X_0}\mathrm{Hilb(\mathscr{S})}$
is a multiplicity-free $T_\ad$-module; its $T_\ad$-weights are the opposites of the spherical roots of $\mathscr{S}$.
\end{corollary}

\begin{proof}
Recall the notation set up in Section~\ref{tangentspace}.
Let $\gamma\in \Sigma(\Delta)$.
Thanks to Proposition~\ref{tgtAB} and Theorem~\ref{alaKostant}, it suffices to prove that
every section $s_\gamma\in H^0(G.v_{\underline \lambda},\mathcal N_{\overline{G.v_{\underline \lambda}}/V^*})$  can be extended to $\overline{G.v_{\underline \lambda}}\subset V^*$ if and only if  $\gamma\in \Sigma$.
This desired assertion will be obtained by applying Proposition~\ref{extension-section}.

Let $\lambda_D$ be such that $\underline{\hat\lambda_D}^\perp=\underline \lambda^\perp$.
Then by the properties of the weights $\lambda_D$ (see Appendix~\ref{propertiesofthemonoid}),
$(\lambda_D,\beta)\neq 0$ for $\beta\in S$ only if $\beta\in \Sigma$.
It follows that $(\lambda_D,\gamma)\neq 0$ only if $\supp \gamma\cap \Sigma\neq\emptyset$.

Moreover, by Axiom $(\mathbf A 1)$ of spherical systems (Definition~\ref{def-sph-syst}), we have:
$c(D,\gamma)\leq 0$ for every $\gamma\in \Sigma\setminus S$.
Consequently, $\gamma$ satisfies the condition of the first assertion of Proposition~\ref{extension-section}, and in turn, $s_\gamma$ extends to a section of $X_0$.

Let $\gamma=\alpha+\alpha'\in \Sigma(\Delta)\setminus \Sigma$.
Note that $\gamma$ is a positive root of $G$ since $(\alpha,\alpha')<0$ by Theorem~\ref{alaKostant}.
To be definite, let thus $(\gamma,\alpha)>0$.
By Lemma~\ref{consequence-descriptionofweights}, $v_\gamma$ can be chosen in $V(\lambda^+)\oplus V(\lambda^-)$ with $\lambda^\pm$ being the weights in $\Delta$ which are non-orthogonal to $\alpha$.
Since $(\gamma, \alpha)>0$, we can assume: $c(\lambda^+,\gamma)>0$.
Finally, since $\alpha,\alpha'\in \Sigma$, there are at least three weights $\lambda_D$ non-orthogonal to $\gamma$ (Lemma~\ref{properties}).
The weights $\gamma$ and $\lambda^+$ thus satisfy the conditions of Proposition~\ref{extension-section}-(ii). This statement implies that $s_\gamma$ can not extend to $X_0$.

Let $\gamma=\alpha\in S\cap \Sigma$.
As explained in the first paragraph of the proof of Proposition~\ref{descriptionofweights}, we can choose $v_\alpha$ to be equal to $X_{-\alpha}v_{\lambda_D}$ for some $\lambda_D$.
Further, $c(D,\alpha)=1$ by Axiom $(\mathbf A 1)$ of spherical systems.
We can thus follow the arguing of the proof of Proposition~\ref{extension-section}-(i) to conclude that $s_\alpha$ extends to $X_0$.
\end{proof}

\begin{theorem}\label{obstructionspace}
The obstruction space $(T^2_{X_0})^{\tilde{G}}$ for the functor ${\mathcal Def}^{\tilde{G}}_{X_0}$ of invariant infinitesimal deformations of $X_0$ is trivial.
\end{theorem}

\begin{proof}
Thanks to Proposition~\ref{obstruction-cotgt}, it suffices to prove that the map $H^1(f)$ displayed therein is injective.
The latter is obtained in Appendix~\ref{injectivity}.
 \end{proof}

\begin{corollary}\label{Hilb}
The scheme $\Hilb(\mathscr{S})$ is isomorphic to an affine space where the
adjoint torus of $G$ acts linearly with
weights equal to the opposites of the spherical roots of $\mathscr{S}$.
\end{corollary}

\begin{proof}
Thanks to Theorem~\ref{obstructionspace}, ${\mathcal Def}^{\tilde{G}}_{X_0}$ is trivial  and in turn $\Hilb(\mathscr{S})$ is smooth by Corollary~\ref{Schlessinger}.
Being also connected (by Theorem~\ref{AB-representability}), $\Hilb(\mathscr{S})$ is irreducible hence consists of a single $T_\ad$-orbit closure.
To sum up, $\Hilb(\mathscr{S})$ is a smooth toric $T_\ad$-variety which is affine (Theorem~\ref{AB-representability}) and has with a single fixed point (Theorem~\ref{ABtoricaction}) hence it is an affine space by Luna's Slice Theorem.
\end{proof}

\begin{corollary}\label{universalfamily}
Let $X$ be a wonderful $G$-variety with total coordinate ring $R(X)$ and spherical system $\mathscr S_X$.
If $\mathscr{S}_X$ is spherically closed then the quotient morphism
$$
\pi: \mathrm{Spec}\,R(X) \rightarrow \mathrm{Spec}\,(R(X)^G)
$$
can be regarded as the universal family of $\Hilb(\mathscr{S}_X)$.
\end{corollary}

\begin{proof}
By Proposition~\ref{Brionfamily}-(ii) together with Proposition~\ref{colors-sphericalsystems-wful}, the fibers over closed points of $\mathrm{Spec}\,R(X)^G$ can be regarded as closed points of $\Hilb(\mathscr{S}_X)$.
This, together with the universal property fulfilled by $\Hilb(\mathscr{S}_X)$, implies the existence of a morphism 
$\iota:\mathrm{Spec}\,R(X)^G\rightarrow\Hilb(\mathscr{S}_X)$. The fibers of $\pi$ being pairwise distinct, the morphism $\iota$ is injective.
But, $\Hilb(\mathscr{S}_X)$ and $\mathrm{Spec}\,R(X)^G$ both equal the affine space of dimension $r$ with $r$ being the number of spherical roots of $X$; see Theorem~\ref{Hilb} and Proposition~\ref{Brionfamily}-(i) respectively.
 The morphism $\iota$ is thus an isomorphism.
\end{proof}

\subsection{The wonderful variety attached to a spherical system}\label{proofofconjecture}
Let $r$ denote the number of spherical roots of $\mathscr{S}$ and
let $X_1$ be a closed point of $\mathrm{Hilb}(\mathscr{S})$ whose $T_\ad$-orbit is dense.

Consider the (universal) family over $\mathrm{Hilb}(\mathscr{S})\simeq\mathbb A^r$
$$
\begin{CD}
\mathcal X^{\mathrm{univ}}  @>{\pi}>> \mathbb A^r .
\end{CD}
$$
Then the coordinate ring $R(\mathscr{S})$ of $\mathcal X^\mathrm{univ}\subset V\times\mathbb A^r$
is isomorphic as a $(\tilde{G}\times T_\ad)$-algebra to
$$
R(\mathscr{S})=\oplus_{\lambda\in\mathbb N\Delta} k[X_1]_{\lambda}e^{\lambda}\otimes k[e^\sigma:\sigma\in\Sigma];
$$
\textsl{see} the recalls made in Section~\ref{Toricaction}.
Here $e^\lambda$ with $\lambda\in\mathbb N\Delta$ stands for the character $e^\omega$ in $k[T]$ whenever $\lambda=(\omega,\chi)$.

According to our previous observations (made before Theorem~\ref{AB-representability}), the $\tilde{G}$-variety $X_1$ is spherical and so is  
the $(\tilde{G}\times T)$-variety $\mathcal X^\mathrm{univ}$. 
In the next statement, we describe, in particular, the spherical roots of $\mathcal X^\mathrm{univ}$;
one may consult Appendix~\ref{Basicmaterialonsphericalvarieties} for terminology and related notation used below (e.g. the definition of the spherical roots of a spherical, non necessarily wonderful, $G$-variety).

\begin{proposition}\label{towardstheinvariants}
\smallbreak\noindent{\rm(i)}\enspace
The set of spherical roots of the $(\tilde{G}\times T)$-variety $\mathcal X^\mathrm{univ}$ coincides with the set $\Sigma$.
\smallbreak\noindent{\rm(ii)}\enspace
The colors of $\mathcal X^\mathrm{univ}$ are indexed by the set $\Delta$.
More specifically,  a color $D$ of $\mathcal X^\mathrm{univ}$ has an equation in $R(\mathscr{S})$ which
can be identified to the weight vector $v_{\lambda_D}e^{\lambda_D}\otimes 1$ of $R(\mathscr{S})$.
\end{proposition}

\begin{proof}
First, by Proposition 2.13 in~\cite{AB} (along with Remark~\ref{connectionwithmoduli}), the monoid $\mathcal M_{X_1}$ of $X_1$ is the
weight monoid of the $T_\ad$-orbit closure of $X_1$. 
Thanks to Corollary~\ref{Hilb}, the submonoid $\mathcal M_{X_1}$ of $\Xi(T)$ is spanned by $\Sigma$.

Finally by the characterization of the spherical roots recalled in~Appendix~\ref{Basicmaterialonsphericalvarieties}
and the fact that $\mathscr{S}$ is spherically closed, the spherical roots of $X_1$ (hence of $\mathcal X^{\mathrm{univ}}$) are exactly the elements of the given set $\Sigma$.

The second assertion is obtained following the procedure stated in Section~\ref{Camus'lemma}, which
enables to derive the set of colors of any affine spherical variety from the spherical roots and the weight monoid of this variety.
In particular, we get that the $(B\times C^\circ\times T)$-weights of the colors of $\mathcal X^\mathrm{univ}$ are given by the $(\lambda_D,\varepsilon_D)$; $\varepsilon_D$ being the $D$-component character of $T$.
\end{proof}

Let $\mathring{\mathcal X}^{\mathrm{univ}}$ be the open subset of $\mathcal X^\mathrm{univ}$ defined as follows
$$
\mathring{\mathcal X}^{\mathrm{univ}}=\{x\in\mathcal X^{\mathrm{univ}}: \tilde{G}.x \mbox{ is open in } \pi^{-1}\pi(x)\}.
$$

\begin{lemma}\label{goodopenset}
\smallbreak\noindent{\rm(i)}\enspace
The elements of $\mathring{\mathcal X}^{\mathrm{univ}}$
are the elements of $\mathcal X^{\mathrm{univ}}\subset V\times \mathbb A^r$
which project non-trivially onto each simple $\tilde{G}$-submodule $V(\lambda_D)^*$ of $V$.
\smallbreak\noindent{\rm(ii)}\enspace
We have:
$\mathring{\mathcal X}^{\mathrm{univ}}=\tilde{G}.(\mathcal X^{\mathrm{univ}}\setminus\cup_{\Delta} D)$.
\end{lemma}

\begin{proof}
Thanks to Theorem~\ref{AB-representability}, every fiber $\pi^{-1}\pi(x)$ is a non-degenerate spherical $\tilde{G}$-subvariety of $V$.
The first assertion of the lemma thus follows.

From the previous proposition,
the equation $f_D$ in $R(\mathscr{S})$ of a color $D\in\Delta$
can be identified to the weight vector $v_{\lambda_D}e^{\lambda_D}\otimes 1$ in $R(\mathscr{S})$.
Take $y=(v,\underline s)\in V\times \mathbb A^r$.
If $f_D(y)\neq 0$ then clearly $v_{\lambda_D}(v)\neq 0$.
Conversely, if $y$ projects non-trivially onto each $V(\lambda_D)^*$ then there exists $g\in \tilde{G}$ such that the projection of $g.v$ on each 
weight space $V(\lambda_D^*)_{\lambda_D^*}$ is not trivial.
Equivalently, $v_{\lambda_D}(g.v)\neq 0$ for each $D\in\Delta$.
The second assertion of the lemma follows.
\end{proof}

Recall that the dominant weights $\lambda_D$ are linearly independent (\textsl{see} Lemma~\ref{enlargeddefinitionofweights}).
The group $GL(V)^{\tilde{G}}$ is thus isomorphic to the algebraic torus $\mathbb G_m^\Delta$ of dimension equal to the cardinality of $\Delta$. 
The $(T_\ad\times C)$-action on $V$ via $t\mapsto (\lambda_D^*(t))_\Delta$ yields naturally a $\mathbb G_m^\Delta$-action  on $V$: the componentwise multiplication.

Note that the open set $\mathring{\mathcal X}^{\mathrm{univ}}$ is $\mathbb G_m^\Delta$-stable.

\begin{theorem}\label{existence}
The quotient 
$$
X(\mathscr{S})=\mathring{\mathcal X}^{\mathrm{univ}}/\mathbb G_m^\Delta
$$
exists and is geometric;
it is a wonderful $G$-variety with spherical system $\mathscr{S}$.
\end{theorem}

\begin{proof}
By means of Lemma~\ref{goodopenset}-(i), we get that $\mathbb G_m^\Delta$ acts freely on $\mathring{\mathcal X}^{\mathrm{univ}}$ whence the existence of a geometric quotient $\mathring{\mathcal X}^{\mathrm{univ}}/\mathbb G_m^\Delta$; \textsl{see} e.g. in~\cite{MF}.

To prove that $X(\mathscr S)$ is a wonderful $G$-variety, we shall apply the criterion recalled in Proposition~\ref{w'fulcriterion}.

First, note that the $G$-variety $X(\mathscr S)$ is spherical since so is  the  $(G\times\mathbb G_m^\Delta)$-variety $\mathring{\mathcal X}^{\mathrm{univ}}$.

Moreover, by the very definition of the action of $T_\ad$ on $\Hilb(\mathscr{S})$ along with Theorem~\ref{ABtoricaction}, $X(\mathscr S)$ has a unique closed $G$-orbit, namely $(G\times\mathbb G_m^\Delta).v_{\Delta^*}/\mathbb G_m^\Delta$.

The inclusion $\mathring{\mathcal X}^{\mathrm{univ}}\subset\oplus_{D\in\Delta} \left( V(\lambda_D)^*\setminus\{0\}\right)\times\mathbb A^r$ (Lemma~\ref{goodopenset}-(i)) gives raise to a $\tilde{G}$-equivariant morphism
$$
\begin{CD}
\mathring{\mathcal X}^{\mathrm{univ}}@>{\varphi}>>\oplus_{D\in\Delta} \left( V(\lambda_D)^*\setminus\{0\}\right).
\end{CD}
$$

The morphisms $\varphi$ and $\pi$ yield a finite (hence proper) morphism from $X(\mathscr S)$ to the multiprojective space $\prod_\Delta \mathbb P\left(V(\omega_D)^*\right)$. The variety $X(\mathscr S)$ is thus complete.

Fix a color $D$ of $\mathring{\mathcal X}^{\mathrm{univ}}$.
From Proposition~\ref{towardstheinvariants}, we know that $D$ is contained in the preimage under $\varphi$ of the hyperplane $(v_{\lambda_D}=0)$.
Further, as observed in the proof of Lemma~\ref{goodopenset},
the $\tilde{G}$-orbits of $(v_{\lambda_D}=0)$ project trivially onto $V(\lambda_D^*)$.
Consequently, $D$ contains no $\tilde{G}$-orbit.
The $(G\times\mathbb G_m^\Delta)$-variety $\mathring{\mathcal X}^\mathrm{univ}$ is thus toroidal and so is $\mathring{\mathcal X}^\mathrm{univ}/\mathbb G_m^\Delta$ as a $G$-variety.

Let $P$ be the parabolic subgroup of $G\times\mathbb G_m^\Delta$ stabilizing the colors $D$ in $\Delta$ and $P^u$ be its unipotent radical. 
Thanks to the Local Structure Theorem (\textsl{see} Theorem 2.3 as well as Proposition 2.4.1 in~\cite{Br97}) applied to the toroidal spherical $(G\times\mathbb G_m^\Delta)$-variety $\mathring{\mathcal X}^\mathrm{univ}$, there exists an affine toric $(T\times\mathbb G_m^\Delta)$-variety $W$ within $\mathring{\mathcal X}^{\mathrm{univ}}\setminus\cup_\Delta D$ such that the natural morphism
$$
P^u\times W\rightarrow \mathring{\mathcal X}^{\mathrm{univ}}\setminus\cup_\Delta D
$$
is an isomorphism.
Note that $\mathring{\mathcal X}^{\mathrm{univ}}\setminus\cup_\Delta D=\mathcal X^{\mathrm{univ}}\setminus\cup_\Delta D$ by Lemma~\ref{goodopenset}-(ii).

Considering the quotient morphism $\mathcal X^{\mathrm{univ}}\rightarrow \mathrm{Spec}(R(\mathscr{S})^G)\simeq \mathbb A^r$, we get
that the $(G\times\mathbb G_m^\Delta)$-stable prime divisors of $\mathcal X^{\mathrm{univ}}$ correspond to the $\mathbb G_m^\Delta$-stable prime divisors of $\mathbb A^r$.
All these divisors are thus principal;
their equations are given by the $e^\sigma$'s with $\sigma\in\Sigma$.
Since $\mathcal X^{\mathrm{univ}}$ is a spherical $(G\times\mathbb G_m^\Delta)$-variety, the colors and
the $(G\times\mathbb G_m^\Delta)$-stable prime divisors of this variety generate its divisor class group; \textsl{see} Section 5.1 in~\cite{Br97}.
Consequently, $\mathcal X^{\mathrm{univ}}\setminus\cup_\Delta D$ is factorial and in turn so are $P^u\times W$ and $W$. 

The affine variety $W$ being toric and factorial, it is smooth (\textsl{see} for instance Proposition 2.4.6 in \cite{CLS}) hence so is $\mathring{\mathcal X}^{\mathrm{univ}}$.
As already observed $\mathbb G_m^\Delta$ acts freely on  $\mathring{\mathcal X}^{\mathrm{univ}}$, the variety
$\mathring{\mathcal X}^{\mathrm{univ}}/\mathbb G_m^\Delta$ is thus smooth also; see~\cite{MF}.

In the remainder of the proof, set $X=X(\mathscr{S})$.
Recall the definition of the spherical system $\mathscr{S}_X$ of $X$ from Section~\ref{wondervar}.
As shown, the closed $G$-orbit of $X$ is $(G\times\mathbb G_m^\Delta).v_\Delta/\mathbb G_m^\Delta$.
By definition, $S^p_X$ is thus given by the simple roots of $G$ which are orthogonal to every $\omega_D$, $D\in\Delta$, that is $S^p$ itself thanks to Lemma~\ref{obviousppties}. 
From the above discussion along with Proposition~\ref{towardstheinvariants}, it follows readily that $\Sigma_X=\Sigma$ and $\mathcal D_X$ can be identified with $\Delta$. Specifically the spherical system of $X$ is $\mathscr{S}$.

\end{proof}

\begin{corollary}\label{conjecture}
Luna's conjecture is true: to any spherical system $\mathscr{S}$ of an adjoint semisimple algebraic group $G$, there corresponds a unique (up to $G$-isomorphism) wonderful $G$-variety whose spherical system is $\mathscr{S}$.
\end{corollary}

\begin{proof}
It suffices to consider spherical systems which are spherically closed; \emph{see} Section 6 in~\cite{Lu01}.
The existence part is given by the previous theorem whereas the uniqueness part stems for Corollary~\ref{universalfamily}.

\end{proof}

\appendix

\section{Basic material on spherical varieties}\label{Basicmaterialonsphericalvarieties}
Let $G$ be a connected reductive algebraic group.
We fix a Borel subgroup $B$of $G$ and a maximal torus $T\subset B$.
The corresponding set of simple roots (resp. dominant weights) is denoted by $S$ (resp. $\Lambda^+$).
We choose a scalar product $(\cdot,\cdot)$, on the real characters $\Xi(T)\otimes_\mathbb Z \mathbb R$ of $T$, invariant under the Weyl group of $G$ relative to $T$.

\subsection{Birational invariants}\label{invariantssphericalvarieties}

Some of the invariants associated to a wonderful $G$-variety can also be assigned to any spherical $G$-variety $X$.

In particular, one defines similarly \textsl{the set $\mathcal D_X$ of colors of $X$} as the set of $B$-stable not $G$-stable prime divisors of $X$.

To introduce the set of spherical roots of an arbitrary spherical $G$-variety, we set up further notation.
Let $\Lambda(X)$ be the set of $B$-weights of the $B$-eigenvectors $\mathbb C(X)^{(B)}$ of the function field $\mathbb C(X)$ of $X$.
Set $\Lambda(X)^*:=\mathrm{Hom}(\Lambda(X),\mathbb Q)$.
Let $\mathcal V_X$ denote the set of $G$-stable prime divisors of $X$.

Define
$$
\rho_X: \mathcal D_X\cup\mathcal V_X\rightarrow \Lambda(X)^*
$$
by setting
$$
\rho_X(D)(\gamma)=v_D(f_\gamma)
$$
with
$v_D$ being the $B$-invariant normalised valuation of $D$ and $f_\gamma$ the $B$-eigenvector of $\mathbb C(X)$ of weight $\gamma$ (uniquely determined up to a scalar since $X$ has a dense $B$-orbit).

In case $X$ is wonderful, the elements of $\mathcal V_X$ are given by the boundary divisors $D_i$ of $X$.
Set $X_B=X\setminus \cup_\mathcal{D_X} D$.
Then $\mathbb C(X)^{(B)}$ (resp. $\Lambda(X)$) is freely generated by the equation $f_i\in k[X_B]$ (resp. spherical roots $\sigma_i$ of $X$) of $X_B\cap D_i$.
And, the $\rho_X (D_i)$'s form a basis of the abelian group $\Lambda(X)^*$ dual to the basis given by the $\sigma_i$'s.

From this perspective, one may define \textsl{the set $\Sigma_X$ of spherical roots} of an arbitrary spherical variety $X$ as follows.
Let $V_X$ be the set of $G$-invariant discrete $\mathbb Q$-valued valuations of $\mathbb C(X)$. 
As shown in~\cite{Br97}, one may regard $V_X$ in $\Lambda(X)^*$ and $V_X$ is a simplicial convex cone in $\Lambda(X)^*$.
The set $\Sigma_X$ is thus defined as the set of primitive linearly independent elements of $\Lambda(X)$ such that $V_X$ is the dual cone of $-\Sigma_X$.

\subsection{Affine case}
Throughout this subsection, $X$ denotes an affine spherical $G$-variety.

\subsubsection{}\label{affinesphericitycriterion}
There exists another characterization of $\Sigma_X$, as follows.

Let $k[X]$ be the algebra of regular functions of $X$.
Let $\Lambda^+(X)$ denote \textsl{the weight monoid of $X$},
that is the submonoid of $\Lambda^+$ given by the highest weights of $k[X]$.
Since $X$ is affine, $\Lambda(X)=\mathbb Z\Lambda(X)^+$. Further, we have the following characterization
\begin{equation}\label{charweightmonoid}
\Lambda^+(X)=\{\gamma\in\Lambda(X): \rho_X(D)(\gamma)\geq 0 \mbox{ for every } D\in\mathcal D_X\cup\mathcal V_X\}.
\end{equation}

First, let us recall the following criterion (\textsl{see}~\cite{Br10} for a survey).
An affine $G$-variety $X$ is spherical if and only if $\Lambda(X)\cap \mathbb Q_+\Lambda^+(X)=\Lambda^+(X)$ and
$k[X]$ is multiplicity-free as a $G$-module, that is, every simple $G$-module occurs in $k[X]$ at most once.

Let $\lambda,\mu$ and $\nu$ be in $\Lambda^+(X)$ and such that $V(\nu)\subset V(\lambda)V(\mu)$ - the product being taken in $k[X]$.
Let $\mathcal M_X$ be the monoid generated by the $\lambda+\mu-\nu$'s in $\Lambda^+\otimes_{\mathbb Z} \mathbb R$.
Then $\mathbb Z\mathcal M_X\cap \mathbb Q_+\mathcal M_X$ is freely generated by a subset of spherical roots of $G$ which is proportional
to the set $\Sigma_X$; \emph{see} Theorem 1.3 in~\cite{K96} and Section 4 in~\cite{Br97}.

\subsubsection{}\label{Camus'lemma}
Let us recall from Section 10.1 in~\cite{Ca}, how one can derive the weights of the colors of $X$ from $\Lambda^+(X)$ along with $\Sigma_X$. 

Like in case of a wonderful variety (\emph{see}~\ref{colorsvswonderful} in the body text), one defines the sets $\mathcal D_X(\alpha)$ for each $\alpha\in S$.
Let $\mathcal D_X(\Sigma_X\cap S)$ be the union of the $\mathcal D_X(\alpha)$'s where $\alpha\in\Sigma_X$.
By \cite{L97}, the set $\mathcal D_X$ is entirely determined by  $\mathcal D_X(\Sigma_X\cap S)$.

Specifically, in case $\alpha\not\in\Sigma_X$, $\mathcal D_X(\alpha)$ contains at most one element and $\mathcal D_X(\alpha)\cap\mathcal D_X(\beta)$ is not empty only if $\alpha+\beta$ is a spherical root of $X$ of type $\mathsf A_1\times\mathsf A_1$.
The $B$-weight of an element of one such $\mathcal D_X(\alpha)$ is thus well-determined.
 
Further, the sets $\Sigma_X$ and $\mathcal D_X(\Sigma_X\cap S)$ satisfy the axioms $\mathbf A$ of a spherical system with $\mathcal D_X(\Sigma_X\cap S)$ playing the role of $\mathbf A$, $\mathcal D_X(\alpha)$ that of $\mathbf A(\alpha)$ and $\rho_X$ that of the pairing $c$.
In particular, if $\alpha\in\Sigma_X$ then $\mathcal D_X(\alpha)$ contains exactly two colors.
Further, by Axiom $(\mathbf A2)$, the whole set $\rho_X (\mathcal D_X(\alpha))$ is determined only by one of its elements.

Finally, let $S^p_X$ be the set of simple roots $\alpha$ such that $\mathcal D_X(\alpha)$ is empty.
Then the triple $(S^p_X,\Sigma_X,\mathcal D_X(\Sigma_X\cap S)$ is a spherical system of $G$; note that this assertion is the part of Theorem 2 in~\cite{Lu01} which is valid for any group $G$.
In particular, by means of Axiom $(S)$ along with $(\ref{charweightmonoid})$, one can extract the colors (hence their $B$-weights) not contained in $\mathcal D_X(\Sigma_X\cap S)$.

To sum up, to obtain $\rho_X (\mathcal D_X)$ from the only data $\Lambda^+(X)$ and $\Sigma_X$, it remains to characterize only one of the elements of each $\rho_X (\mathcal D_X(\alpha))$ for $\alpha\in \Sigma_X\cap S$. This step is achieved by the following statement.

\begin{lemma}[Lemma~10.1 in~\cite{Ca}]
Let $\alpha\in S\cap\Sigma_X$.
Then one of the elements of $\rho_X (\mathcal D_X(\alpha))$ defines a face of $\Lambda(X)^+$.
\end{lemma}

As a sake of convenience, let us recall Camus' proof of this lemma.

\begin{proof}
We proceed by contradiction.
Then by the characterization (\ref{charweightmonoid}) of $\Lambda^+(X)$, we have
$\Lambda^+(X)=\{\gamma\in\Lambda(X): \rho_X(D)(\gamma)\geq 0 \mbox{ for every }D\in\mathcal V_X\cup\mathcal D_X\setminus\mathcal D_X(\alpha) \}$.
On the other hand, by Axiom $(\mathbf A 2)$, we have $\rho_X(D)(\alpha)\leq 0$ for every color $D$ of $X$ not in $\mathcal D_X(\alpha)$.
Further, since $\alpha\in\Sigma_X$, the inequality $\rho_X(D)(\alpha)\leq 0$ holds for every $D\in\mathcal V_X$; recall the definition of $\Sigma_X$ given in the previous subsection.
Accordingly, we get $\gamma\in \Lambda^+(X)$ but this is absurd since $-\alpha$ is not dominant. The lemma follows.
\end{proof}

\section{Combinatorics related to spherical systems}\label{combinatorics}
Given any linear combination $\beta=\sum n_\alpha \alpha$ of simple roots, the support $\supp\beta$ of $\beta$ is defined as the set of simple roots $\alpha$ such that $n_\alpha\neq 0$.

\subsection{Spherical roots: list and properties}\label{listsphericalroots}

In the table below, we recollect from~\cite{W} the list of spherical roots.

A spherical root is by definition a spherical root of some group $G$.
The latter appears in loc. cit. to be either a positive root of $G$ or a sum of two positive roots of $G$.
The support of any given spherical root $\sigma$ thus defines a root system whose type is referred below as the type of $\sigma$;
we label the simple roots of this root system according to Bourbaki notation. Note that these simple roots are also simple roots of $G$.

\bigbreak

\begin{center}
\begin{tabular}{l | l}
\firsthline
Type of $\sigma$ & $\sigma$\\
\hline
$\mathsf A_1\times\mathsf A_1$ &$\alpha_1+\alpha_1'$ \vspace{0.1cm}\\
\hline
$\mathsf A_r$ &$\alpha_1+\ldots+\alpha_r$\\
& $2\alpha_1$ ($r=1$)\\
&$\alpha_1+2\alpha_2+\alpha_3$\vspace{0.1cm} ($r=3$)\\
\hline
$\mathsf B_r$, $r\geq2$&$\alpha_1+\ldots+\alpha_r$\\
&$2\alpha_1+\ldots+2\alpha_r$\\
&$\alpha_1+2\alpha_2+3\alpha_3$ ($r=3$)\vspace{0.1cm}\\
\hline\vspace{0.1cm}
$\mathsf C_r$, $r\geq3$&$\alpha_1+2\alpha_2+\ldots+2\alpha_{r-1}+\alpha_r$\vspace{0.1cm}\\
\hline
$\mathsf D_r$, $r\geq4$&$2\alpha_1+\ldots+2\alpha_{r-2}+\alpha_{r-1}+\alpha_r$\vspace{0.1cm}\\
\hline
$\mathsf F_4$&$\alpha_1+2\alpha_2+3\alpha_3+2\alpha_4$\vspace{0.1cm}\\
\hline
$\mathsf G_2$&$\alpha_1+\alpha_2$\\
&$2\alpha_1+\alpha_2$\\
&$4\alpha_1+2\alpha_2$\\
\lasthline
\end{tabular}
\end{center}

\begin{lemma}
\label{pptiesphericalroots}
\smallbreak\noindent{\rm(i)}\enspace
Let $\sigma$ be a spherical root of $G$ and $\alpha$ be in the support of $\sigma$.
Then $(\sigma,\alpha)\geq 0$ unless $\alpha=\alpha_1$ and $\sigma=\alpha_1+\alpha_2$ is of type $\mathsf G_2$.
Further, if $(\sigma,\alpha)>0$ then $\sigma-\alpha$ is a root of $G$.
\smallbreak\noindent{\rm(ii)}\enspace
Let $\sigma$ and $\alpha$ be two spherical roots of the same spherical system of $G$.
Suppose $\alpha\in\supp\gamma$. Then either $\sigma$ is of type $\mathsf A_1$, $\mathsf C_r$ or $\mathsf G_2$. 
In the latter case, $\alpha$ and $\sigma$ are as in (i).
\end{lemma}

\begin{proof}
The statement (i) follows readily  from the list of spherical roots.
To prove (ii), we shall make use of Lemma~\ref{compatibilitycondition}.
Let $(S^p, \Sigma,\mathbf A)$ be the spherical system under consideration.
First note that if $\sigma\not\in S$ then the axioms $(\mathbf A2)$ and $(\mathbf A3)$ imply that $(\sigma,\alpha)\leq 0$ and in turn $(\sigma,\alpha)=0$ by (i), unless $\sigma$ is of type $\mathsf G_2$ as stated therein.
Thanks to the compatibility condition fulfilled by  $(S^p,\alpha)$, we get that the simple roots adjacent to $\alpha$ do not belong to $S^p$.
We now apply the compatibility condition fulfilled by  $(S^p,\sigma)$.
We thus obtain  that  the support of $\sigma$ is of cardinality at most $2$ or $\sigma$ is of type $\mathsf C_r$.
We get also that,  in type $\mathsf A_2$ and $\mathsf G_2$ with $\sigma$ not as in (i), the root $\alpha$ belongs to $S^p$ since $(\sigma,\alpha)=0$.
Therefore these cases have to be ruled out. The lemma follows.
\end{proof}

\begin{definition}[\cite{BP}, 1.1.6]
Let $S^p$ be a subset of $S$ and $\sigma$ be a spherical root of $G$.
The couple $(S^p,\sigma)$ is \textsl{compatible} if
$$
S^{pp}(\sigma)\subset S^p\subset S^p(\sigma)
$$
where $S^p(\sigma)$ is the set of simple roots orthogonal to $\sigma$ and
$S^{pp}(\sigma)$ is one of the following sets
\smallbreak\noindent
{\rm -}\enspace $S^p(\sigma)\cap\supp(\sigma)\setminus\{\alpha_r\}$ if $\sigma=\alpha_1+\ldots+\alpha_r$ is of type $\mathsf B_r$,
\smallbreak\noindent
{\rm -}\enspace $S^p(\sigma)\cap\supp(\sigma)\setminus\{\alpha_1\}$ if $\sigma$ is of type $\mathsf C$,
\smallbreak\noindent
{\rm -}\enspace $S^p(\sigma)\cap\supp(\sigma)$ otherwise.
\end{definition}

\begin{lemma}[loc.cit.]\label{compatibilitycondition}
Let $S^p$ be a subset of $S$ and $\sigma$ a spherical root of $G$.
Then the data $S^p$ and $\sigma$ are those associated to a wonderful $G$-variety if and only if $(S^p,\sigma)$ is compatible.
\end{lemma}

By Axiom $(S)$ in the definition of a spherical system, we shall refer to the aforementioned compatibility condition as well as Axiom $(S)$.

\begin{definition}
\smallbreak\noindent
{\rm(i)}\enspace
A spherical root $\sigma$ of $G$ is \textsl{loose} if $\sigma$ equals either the root $\alpha_1+\ldots+\alpha_n$ of type $\mathsf B_n$ or the root $2\alpha_1+\alpha_2$ of type $\mathsf G_2$.
\smallbreak\noindent
{\rm(ii)}\enspace
A spherical system is \textsl{spherically closed} if it does not contain any loose spherical root.
\end{definition}

\begin{remark}
The above definition of a spherically closed spherical system coincides with that stated in Definition~\ref{sphericallyclosed}.
\end{remark}

\begin{lemma}\label{lindependencesphroots}
The spherical roots of any spherical system of $G$ are linearly independent characters in $\Lambda(T)\otimes_{\mathbb Z}\mathbb R$.
\end{lemma}

\begin{proof}
We shall show that $(\sigma,\sigma')\leq 0$ for every pairwise distinct spherical roots $\sigma,\sigma'$ of a given spherical system of some group $G$.
Since any spherical root is a sum of positive roots of $G$, there exists a dominant weight non-orthogonal to all spherical roots under consideration.
The lemma thus follows from a well-known result about vectors in an Euclidian space with non-positive pairings.

Fix $\sigma$ and $\sigma'$ as above. First note that if one of these spherical roots is equal to $2\alpha$ (of type $\mathsf A_1$), then the desired inequality is given by Axiom $(\Sigma 1)$. We shall exclude this case in the remainder.
If the supports of $\sigma$ and $\sigma'$ do not intersect, then $(\sigma,\sigma')\leq 0$ obviously holds.
Then let $\alpha$ be in both supports. One may suppose further that $\alpha\not\in S^p$; otherwise the lemma is clear thanks to Axiom $(S)$.
By Lemma~\ref{pptiesphericalroots}, we have $(\sigma,\alpha)\geq 0$.

If $(\sigma,\alpha)=0$ then $\sigma$ is of type either $\mathsf B_r$ or $\mathsf C_r$. Further, by Axiom $(S)$ we have $\alpha=\alpha_r$ and $\alpha=\alpha_1$ respectively together with $\sigma=\alpha_1+\ldots+\alpha_r$ (up to $2$) in the former case.
Suppose that $\alpha=\alpha_r$; the other case may be handled similarly.
Since $\alpha$ lies in the support of $\sigma'$ also, we get from the above table that $\sigma'$ is of type either $\mathsf B_r$ or $\mathsf A_1\times\mathsf A_1$. In the latter case, we have $(\sigma,\sigma')=(\sigma,\alpha+\alpha')$ which equals $2(\sigma,\alpha)$ by Axiom $(\Sigma 2)$ hence it is $0$ by assumption.
If both $\sigma$ and $\sigma'$ are of type $\mathsf B_r$ then their supports are included in one another. This is impossible by Axiom $(S)$ and the assumption of non-proportionality on the set of spherical roots.

It remains to consider the case where $(\sigma,\alpha)>0$ and $(\sigma',\alpha)>0$ for all $\alpha$ in $\supp\sigma\cap\supp\sigma'$.
First, note that neither $\sigma$ nor $\sigma'$ is a simple root of  $G$.
Since $\sigma$, $\sigma'$ can not be proportional, we get from the table along with $(S)$ that their supports are not equal.
To be definite, let $\beta$ be in the support of $\sigma$ but not in that of $\sigma'$.
We can choose (and we do) $\beta$ such that it is not orthogonal to $\supp\sigma\cap\supp\sigma'$ hence $\beta$ is adjacent to some $\alpha$ as above.

We have: $(\sigma,\beta)\geq 0$ and also $(\sigma',\beta)<0$ hence by Axiom $(S)$, $\beta\not\in S^p$.
If $(\sigma,\beta)>0$ we get $\sigma=\alpha+\beta$ (of type $\mathsf A_2$). Applying again Axiom $(S)$, we obtain that $\sigma'$ is a root of $G$
equals to $\alpha+\beta'$ with $\beta'\neq\beta$ a simple root of $G$. The desired inequality follows obviously.
The case $(\sigma,\beta)=0$ may be worked out by similar arguments.
\end{proof}

\subsection{Properties of the weights attached to a spherical system}\label{propertiesofthemonoid}

Given a simple root $\alpha$ of $G$, we denote its associated fundamental weight by $\omega_\alpha$ and its co-root by $\alpha^\vee$.
Recall that $\alpha^\vee$ corresponds to $2\alpha/(\alpha,\alpha)$ under the identification of the Lie algebra of $T$ and its dual via the Killing from $(\cdot,\cdot)$ of $G$.

Fix a spherical system $\mathscr{S}=(S^p,\Sigma,\mathbf A)$ of $G$ and let $\Delta$ be its set of colors.

If $\alpha$ and $\beta$ are orthogonal simple roots whose sum is an element of $\Sigma$, write $\alpha\sim\beta$.
Recall that 
$$
\Delta=\mathbf A\cup S'/\sim \quad\mbox{ with }\quad S'=S\setminus((S\cap\Sigma)\cup S^p).
$$
In the remainder, we denote the class of $\alpha\in S'$ in $S'/\sim$ by $D_\alpha$.

We shall abuse notation by denoting the set of weights 
$$
\lambda_D=(\omega_D,\chi_D)\quad \mbox{with $D\in \Delta$}
$$
associated to $\mathscr{S}$ also by $\Delta$.

Recall that the $\lambda_D$'s are weights of $T\times C$ with $C$ being the diagonalizable group associated to $\mathscr S$.
Further, the weights $\omega_D$ are defined as follows: 
$$
\omega_D=\left\{
\begin{array}{ll}
2\omega_\alpha &  \mbox{ if $D=D_\alpha$ with $2\alpha\in \Sigma$}\\
\sum_\alpha\omega_\alpha  &  \mbox{ with $D\in\mathbf A(\alpha)$ or $D=D_\alpha$ with $\alpha\in S'\setminus\frac{1}{2} \Sigma$}\\
\end{array}\right ..
$$

Given $\lambda\in \Delta$, we shall write $\omega_\lambda$ instead of $\omega_D$ if $\lambda=\lambda_D$.
With this convention, we set 
$$
(\lambda,\alpha):=(\omega_\lambda,\alpha).
$$

The following lemma gathers some straightforward properties of the elements of $\Delta$.
 
\begin{lemma}\label{obviousppties}
\begin{enumerate}
	\item $\alpha\in S^p$ if and only if $(\lambda,\alpha)=0$ for all $\lambda\in\Delta$.
	\item Let $\alpha\in S\setminus S^p$. Then there are at most two $\lambda\in\Delta$ non-orthogonal to $\alpha$.
	There exists a unique one if and only if $\alpha\not\in \Sigma$.
	\item Suppose $(\lambda,\alpha)\cdot(\lambda,\alpha')\neq 0$. Then either $\alpha\sim\alpha'$ or $\alpha,\alpha'\in\Sigma$.
	In particular, $\alpha\in\Sigma$ if and only if $\alpha'\in\Sigma$.
	\item If $\mathbf A=\emptyset$ then the weights $\omega_D$ are linearly independent.
\end{enumerate}
\end{lemma}

\begin{proof}
We need to prove only the last item; the other assertions follow readily from the very definition of the set of colors $\Delta$ and the weights $\omega_D$.
Then let $\mathbf A=\emptyset$. Note that $\Delta$ can be identified with $(S\setminus S^p)/\sim$ and that the assignment $D\mapsto\omega_D$ yields a bijection between the set of colors $\Delta$ and the set of weights $\omega_D$.
Thanks to (2), we have further:
$$
(\lambda,\alpha)\cdot(\lambda',\alpha)=0\quad\mbox{ for all $\lambda\neq\lambda'\in\Delta$ and all $\alpha\in S$}.
$$
This implies the linear independence of the weights $\omega_D$.
\end{proof}

\begin{lemma}\label{properties}
The weights $\lambda$ in $\Delta$ satisfy the following properties (for any disctinct $\alpha,\alpha'\in S$).
\begin{enumerate}
	\item $(\lambda,\alpha^\vee)\leq 2$;
	\item if $(\lambda,\alpha^\vee)=2$ then $\omega_\lambda=2\omega_\alpha$ and $(\lambda',\alpha)=0$ for all $\lambda'\neq \lambda$ in $\Delta$;
	\item if $(\lambda,\alpha)\cdot(\lambda,\alpha')\neq 0$ then either $\Delta$ is a singleton or there exists $\lambda'\in\Delta$ such that
	$(\lambda',\alpha)\cdot(\lambda',\alpha')=0$;
	\item if $(\lambda,\alpha)\cdot(\lambda',\alpha)\neq 0$ with $\lambda'\neq\lambda$ in $\Delta$ then $(\lambda,\alpha')\neq 0$  implies that  $(\lambda',\alpha')=0$.
\end{enumerate}
\end{lemma}

\begin{proof}

The first item as well as the first assertion of the second item follow readily from the definition of the $\omega_D$'s.
In particular, if $(\lambda,\alpha^\vee)=2$ for $\lambda\in\Delta$ then $2\alpha\in\Sigma$. Thanks to the aforementioned lemma, we obtain the second item.

Now if $(\lambda,\alpha)\cdot(\lambda,\alpha')\neq 0$ then  either $\alpha\sim\alpha'$ or both $\alpha$ and $\alpha'$ belong to $\Sigma$ (by Lemma~\ref{obviousppties}).
If $\alpha\sim\alpha'$ then $(\lambda',\alpha+\alpha')=0$ for all $\lambda'\neq \lambda$ in $\Delta$, thanks to (2) of Lemma~\ref{obviousppties}.
In case $\alpha,\alpha'\in\Sigma$, suppose there exists $\lambda'\in\Delta$ different from $\lambda$ and non-orthogonal to $\alpha$.
By definition of the $\omega_D$'s, one may thus assume (to be definite) that $\lambda=\lambda_{D_\alpha^+}$ and $\lambda'=\lambda_{D_\alpha^-}$.
Since $\lambda$ is assumed to be non-orthogonal to $\alpha'$, we have further: $c(D_\alpha^+,\alpha')=1$.
In another hand, applying Axiom $(\mathsf{A}_3)$ with $\sigma=\alpha'$, we get: $c(D_\alpha^-,\alpha')\neq 1$ and in turn $(\lambda',\alpha')=0$.
 
Finally, again by Lemma~\ref{obviousppties}-(2), if $(\lambda,\alpha)\cdot(\lambda',\alpha)\neq 0$ with $\lambda\neq\lambda'$ then $\alpha\in\Sigma$. By the arguments used to prove the third item, we obtain the last assertion of the lemma. 
\end{proof}

\begin{lemma}\label{additionalppty}
Let $\alpha$, $\alpha'$, $\delta$ and $\delta'$ be pairwise distinct simple roots not in $S^p$.
Suppose $(\alpha,\alpha')\cdot(\delta,\delta')\neq 0$ and $(\delta,\alpha)=0$.
If $\Delta$ contains more than two elements then one of its elements is orthogonal to $\alpha+\alpha'$.
\end{lemma}

\begin{proof}
We proceed by contradiction: Suppose $(\lambda,\alpha+\alpha')\neq 0$ for all $\lambda\in\Delta$.

By Lemma~\ref{obviousppties}-(1), there exists an element $\lambda_\delta\in\Delta$ which is not orthogonal to $\delta$.
By hypothesis, $(\lambda_\delta,\alpha)\neq 0$ or $(\lambda_\delta,\alpha')\neq 0$.

Suppose first that neither $\alpha$ nor $\alpha'$ are in $\Sigma$.
To be definite, let $(\lambda_\delta,\alpha)\neq 0$ then by Lemma~\ref{obviousppties}-(3), $\alpha~\sim\delta$.
Similarly, one gets that $\alpha'\sim\delta'$.
The set $\Delta$ thus contains only $\lambda_\delta$ and $\lambda_{\delta'}$ -- whence the contradiction.

Let $\alpha\in\Sigma$ now. If $(\lambda_\delta,\alpha)\neq 0$ then $\delta\in\Sigma$.
If $(\lambda_\delta,\alpha')\neq 0$ then either $\alpha'\sim\delta$ or $\alpha',\delta\in\Sigma$ by Lemma~\ref{obviousppties}-(3).
Since $(\alpha,\alpha')\neq 0=(\alpha,\delta)$, Axiom $(\Sigma 2)$ prevents  $\alpha'+\delta$ from being in $\Sigma$ hence $\alpha',\delta\in\Sigma$.  
Now by Lemma~\ref{obviousppties}-(2), there exists $\lambda_\delta'\neq\lambda_\delta\in\Delta$ non-orthogonal to $\delta$.
By Lemma~\ref{properties}-(4),  $\lambda_\delta$ or  $\lambda_\delta'$ has to be orthogonal to $\alpha$ hence non-orthogonal to $\alpha'$ by assumption. We get in turn that $\alpha'\in\Sigma$.
Analogously, one proves that $\delta'\in\Sigma$.
Finally, the orthogonality of $\alpha$ and $\delta$ together with Axiom $(\mathbf A 3)$ yield the contradiction. 

If $\alpha'\in\Sigma$, one proves similarly that $\alpha,\delta,\delta'\in\Sigma$ and concludes as previously.
\end{proof}

\section{Computations of cohomology groups}\label{cohomologies}

In this appendix, we assume the reader is familiar with the notions and properties of spherical systems, their colors and their weights.
If not, please consult Section~\ref{Spherical systems} and Appendix~\ref{combinatorics}.

Let $G$ be a simply connected semisimple complex algebraic group and $T$ be a maximal torus of $G$.
We denote the set of roots of $G$ relative to $T$ by $\Phi$.
Fix a Borel subgroup $B$ of $G$ containing $T$ and let $S$ (resp. $\Lambda^+$) be the set of simple roots of $G$ relative to $B$ and $T$ (resp. of dominant weights).

For each $\alpha\in\Phi$, we choose a root vector $X_\alpha$, that is a $T$-weight vector of weight $\alpha$ in the Lie algebra $\mathfrak g$ of $G$.
We denote the co-root of $\alpha$ by $\alpha^\vee$.
Recall that $\alpha^\vee$ corresponds to $2\alpha/(\alpha,\alpha)$ under the identification of the Lie algebra of $T$ and its dual using the Killing form $(\cdot,\cdot)$ of $G$.

Given a finite dimensional $G$-module $V$, a weight vector (resp. the weight space) in $V$ of  weight $\mu$ ($\mu$ a character of $T$) is denoted by $v_\mu$ (resp. $V_\mu$).

Let
$$
V=\oplus_{i=1}^s V(\lambda_i)
$$
be the decomposition of $V$ into irreducible $G$-modules.
We set
$$
v_{\underline\lambda}=v_{\lambda_1}+\ldots+v_{\lambda_s}.
$$

Let $T_\ad$ be the adjoint torus of $G$, namely the quotient of $T$ by the center of $G$.
In the sequel, we consider the normalized action of $T_\ad$ on $V$, that is the action of $T_\ad$ naturally induced by setting
$$
t.v=\lambda_i(t)t^{-1} v\quad \mbox{if $v\in V(\lambda_i)$}.
$$

Remark that the $T_\ad$-weights of $V$ are thus the opposites of those of the dual $V^*$ acted on by $T_\ad$ as stated in Section~\ref{Toricaction}.

In the remainder, we fix a spherical system $\mathscr{S}=(S^p,\Sigma,\mathbf A)$ of $G$.
Let $C$ be the diagonalizable subgroup associated to $\mathscr{S}$ and $C^\circ$ be its identity-component.
We denote the set of weights (or equivalently of colors) of $\mathscr{S}$ by $\Delta$. Let $\lambda_D$ be the weight associated to the color $D\in\Delta$.

Set
$$
\tilde{G}=G\times C^\circ
$$
and
$$
V(\Delta)=\oplus_{D\in\Delta} V(\lambda_D),
$$
that is the $\tilde{G}$-module whose highest weights are the $\lambda_D$'s.
Further, let
$$
v_\Delta=v_{\underline\lambda}\quad\mbox{ if $V=V(\Delta)$}.
$$

\subsection{Auxiliary lemmas}

For convenience, we shall recall the following statements.

\begin{lemma}(\cite[Proposition 3.4]{BC1})\label{rootsupport}
Suppose that the dominant weights $\lambda_1,\ldots,\lambda_s$ are linearly independent and 
generate a monoid $\Gamma$ such that $\mathbb Z\Gamma\cap \Lambda^+=\Gamma$.
Let $\gamma$ be a $T_\ad$-weight of $(V/\mathfrak g.v_{\underline{\lambda}})^{G_{v_{\underline{\lambda}}}}$.
If $\delta\in\supp\gamma$ such that $\gamma-\delta\not\in\Phi$ then $(\gamma,\delta)\geq 0$.
Further if $(\gamma,\delta)=0$ then $(\lambda_i,\delta)=0$ for every $\lambda_i$.
\end{lemma}

\begin{remark}\label{remarksaturated}
\smallbreak
\noindent{\rm(i)}\enspace
The weights associated to a spherical system of $G$ whose third datum is the empty set fulfill the assumptions of the preceding lemma; \textsl{see} Lemma~\ref{obviousppties}.
Specifically, the second and the last assertions of this lemma imply that $\mathbb Z\Gamma\cap \Lambda^+=\Gamma$.
\smallbreak
\noindent{\rm(ii)}\enspace We will generalize Lemma~\ref{rootsupport} in Proposition~\ref{supportofGamma}.
\end{remark}

\begin{lemma}\label{restriction}
Let $\gamma$ belong to the $\mathbb Z$-span of the $\lambda_i$'s.
Let $L\supset T$ denote the Levi subgroup associated to the set $\supp\gamma$
and $W$ be the $L$-submodule of $V$ generated by $v_{\underline\lambda}$.
Then as $T_\ad$-modules,
$$
(V/\mathfrak g.v_{\underline\lambda})^{G_{v_{\underline\lambda}}}_\gamma
\cong(W/\mathfrak l.v_{\underline\lambda})^{L_{v_{\underline\lambda}}}_\gamma .
$$
\end{lemma}

\begin{proof}
This is a slight generalization of Lemma 3.5 in~\cite{BC1}.
\end{proof}

\begin{lemma}\label{weightvector}
Keep the assumptions of the preceding lemma.
\smallbreak\noindent
{\rm (i)}\enspace
The $T_\ad$-module $(V/\mathfrak g.v_{\underline{\lambda}})^{G_{v_{\underline{\lambda}}}}$ is multiplicity-free and its set $\Sigma(\underline\lambda)$
of $T_\ad$-weights is a set of non-loose spherical roots of $G$.
\smallbreak\noindent
{\rm (ii)}\enspace
The set of simple roots orthogonal to $\lambda_1+\ldots+\lambda_s$ together with $\Sigma(\underline\lambda)$ and the empty set form a spherical system of $G$.
 \smallbreak\noindent
{\rm (iii)}\enspace
Let $x$ be a $T_\ad$-weight vector of $(V/\mathfrak g.v_{\underline{\lambda}})^{G_{v_{\underline{\lambda}}}}$.
Then one of the representatives $v\in V$ of $x$ can be taken in some simple $G$-submodule $V(\lambda)$ of $V$ and such that
$$
[v]\in (V(\lambda)/\mathfrak g.v_{\lambda})^{G_{v_{{\lambda}}}}\quad\mbox{ or }\quad v=X_{-\gamma}v_\lambda.
$$
In particular, if $v=X_{-\gamma}v_\lambda$ then the $T_\ad$-weight $\gamma$ of $x$ is a root
and $(V(\lambda)/\mathfrak g.v_{\lambda})^{G_{v_{{\lambda}}}}$ is trivial.
\end{lemma}

\begin{proof}
The first assertion of the preceding lemma gathers the assertions of Theorem 3.1 and Theorem 3.10 in~\cite{BC1}; the second one is stated in Theorem 4.1 in loc. cit.
whereas the last assertion was obtained while proving Theorem 3.10 in loc. cit.
\end{proof}

\begin{lemma}\label{saturatedcase}
Let $\gamma$ be a spherical root of $G$.
Suppose $\gamma$ is neither a loose spherical root nor a simple root of $G$ and consider $S'$ such that $(S',\gamma)$ is compatible.
Further, let $V$ be the $G$-module whose highest weights are the dominant weights $\omega_D$ associated to the spherical system $(S',\{\gamma\},\emptyset)$.
Then $\gamma$ is a $T_\ad$-weight of $(V/\mathfrak g.v_{\underline\lambda})^{G_{v_{{\underline\lambda}}}}$.
\end{lemma}

\begin{proof}
Note that $(S',\{\gamma\},\emptyset)$ is indeed a spherical system since $\gamma\not\in S$.

Recall that $\gamma$ lies in the $\mathbb Z$-span of the $\omega_D$'s by Lemma~\ref{enlargeddefinitionofweights}.
One may thus assume that the support of $\gamma$ coincides with the whole set $S$ by Lemma~\ref{restriction}.
By the very definition of the weights $\omega_D$ together with the compatibility condition $(S)$, one sees that there are at most two such weights.

Assume first there are two weights $\omega_D$ and $\omega_{D'}$.
Then these weights are fundamental (Lemma~\ref{obviousppties}).
Further, from the table of spherical roots, one can derive the following properties: $\gamma$ is a root, $\gamma\pm\alpha$ is not a root whenever $\alpha\in S^p$ and finally  if $\alpha\not\in S^p$ and
$\gamma$ is not of type $\mathsf C_r$ then the support of $\gamma-\alpha$ does not contain $\alpha$.
It follows that if $\gamma$ is not of type $\mathsf C_r$ then $X_{-\gamma}v_{\omega_D}$ (or $X_{-\gamma}v_{\omega_{D'}}$) gives raise obviously to a $T_\ad$-weight vector in $(V/\mathfrak g.v_{\underline\lambda})^{G_{v_{{\underline\lambda}}}}$.
We postpone the $\mathsf C_r$-case to the very end.

In case of a single weight $\lambda=\omega_D$, we fall in the setting considered by Jansou in~\cite{Js}.
Thanks to Proposition 1.6 in loc. cit., $(V/\mathfrak g.v_{\lambda})^{G_{v_{{\lambda}}}}$ is one-dimensional.
The action of $T$ on $V$, induced by $G$, descends to $(V/\mathfrak g.v_{\lambda})^{G_{v_{{\lambda}}}}$. 
By Proposition 1.8 in loc. cit., the corresponding weight is either $0$ or $-\lambda$; Proposition 1.9 in loc. cit.
identifies precisely the weight according to $\lambda$.
Equivalently, the corresponding $T_\ad$-weight of the normalized action is either $\lambda$ or $2\lambda$.
Again by the definition of the weight $\omega_D$, one sees that $\gamma$ is either equal to $\lambda$ or $2\lambda$.

We are left with the following case: $\gamma$ is of type $\mathsf C_r$ and there are two weights $\omega_D$. Note that these weights are the fundamental weights $\omega_1$ and $\omega_2$.
Then taking $\lambda=\omega_2$, we get as just recalled a $T_\ad$-weight vector $[v_\gamma]$ of weight $\gamma$ in $(V(\lambda)/\mathfrak g.v_{\lambda})^{G_{v_{{\lambda}}}}$.
Take $v_\gamma$ in $V(\lambda)_{\lambda-\gamma}$.
Then easy computations show that $v_\gamma$ gives raise to a $T_\ad$-weight vector in $(V/\mathfrak g.v_{\underline\lambda})^{G_{v_{{\underline\lambda}}}}$ with $V=V(\omega_1)\oplus V(\omega_2)$.
\end{proof}

The following proposition is the announced generalization of Lemma~\ref{rootsupport}.

\begin{proposition}\label{supportofGamma}
Let $V=V(\Delta)$ and  $[v_\gamma]\in V/\mathfrak g.v_{\underline\lambda}$ be a $T_\ad$-weight vector of weight $\gamma$.
Let $\alpha,\delta\in S$ be orthogonal with $\delta$ in the  support of $\gamma$.
Suppose that $\gamma-\delta\not\in\Phi$ and that
$$
[X_\beta v_\gamma]=0\quad\mbox{ for all positive root $\beta$ different from $\alpha$.}
$$
If further $\gamma\in\mathbb Z \Delta$ or $\gamma+\alpha\in\mathbb Z \Delta$ then $(\gamma,\delta)\geq 0$.
Moreover, if $(\gamma,\delta)=0$ then $( \lambda,\delta)=0$ for every $\lambda\in\Delta$.
\end{proposition}

We shall make use of the following lemmas to prove this proposition.

\begin{lemma}\label{non-highest}
Under the assumptions of Proposition~\ref{supportofGamma},
there exists a positive root $\nu\in\Phi$ different from $\alpha$ such that $X_\nu v_\gamma$ is not trivial.
\end{lemma}
 
\begin{proof}
Since $[v_\gamma]\neq 0$, $v_\gamma$ is not a highest weight vector in $V$ and in turn
there exists $\beta\in\Phi$ positive  such that $X_\beta v_\gamma\neq 0$.
In case  $X_\beta v_\gamma\neq 0$ only if $\beta=\alpha$, the weight vector $v_\gamma$ is a linear combination of 
vectors of shape $X_{-\alpha}^{r} v_{\lambda}$ ($\lambda\in\Delta, r\in\mathbb N$). 
The weight $\gamma$ thus equals $\alpha$ (up to a scalar); this contradicts the existence of $\delta\in\supp\gamma$ with $\delta\neq\alpha$ made in the assumptions.
\end{proof}

\begin{lemma}~\label{weightshape}
Under the assumptions of Proposition~\ref{supportofGamma},
the following holds.
\begin{enumerate}
	\item $\gamma$ is a sum of two positive roots.
	\item  The support of $\gamma$ contains a simple root adjacent to $\delta$.
\end{enumerate}
\end{lemma}

\begin{proof}
The first assertion follows readily from the previous lemma.

To prove the second assertion, one may notice that if the supports of $\gamma-\delta$ and $\delta$ are orthogonal then $\gamma$ can be written as a sum 
of two positive roots only if $\gamma-\delta$ itself is a root - whence a contradiction.
\end{proof}

\begin{lemma}~\label{goodrepresentative}
Let $[v_\gamma]\in V/\mathfrak g.v_{\underline\lambda}$ be a $T_\ad$-weight vector  of weight $\gamma$ with $v_\gamma\in \oplus_{\lambda_i} V(\lambda_i)_{\lambda_i-\gamma}$.
Suppose there exists a positive root $\beta$ such that $(\gamma,\beta)>0$ and $X_\beta v_\gamma\in\mathfrak g.v_{\underline\lambda}$. 
Then for all $\lambda_i$  orthogonal to $\beta$, the $\lambda_i$-component $v_\gamma^{\lambda_i}$ of  $v_\gamma$ is equal to $X_{-\gamma}v_{\lambda_i}$ (up to a scalar independent of $\lambda_i$).
In particular, on may choose $v_\gamma$ such that $v_\gamma^{\lambda_i}=0$ for all such $\lambda_i$.
\end{lemma}

\begin{proof}
Let $\lambda=\lambda_i$ be orthogonal to $\beta$. 
First note that $(\lambda-\gamma,\beta)<0$. 
It follows that $X_\beta v_\gamma^\lambda\neq 0$ whenever $v_\gamma^\lambda\neq 0$.
By hypothesis, we have also that $X_\beta v_\gamma= X_{-\gamma+\beta}v_{\underline\lambda}$ (up to a scalar).
Therefore,  if $v_\gamma^{\lambda_j}$ is trivial for some $\lambda_j$ orthogonal to $\beta$ then so are all the other $\lambda_k$-components of $v_\gamma$ with $\lambda_k$ orthogonal to $\beta$.
Moreover we get from the aforementioned equality that
$X_{-\beta}X_\beta v_\gamma= X_{-\beta}X_{-\gamma+\beta}v_{\underline\lambda}$ (up to a scalar).
The $\lambda$-component of the left hand side  equals $v_\gamma^\lambda$ (up to a scalar $a_\lambda$) whereas that of  the right hand side equals $0$ or $X_{-\gamma} v_\lambda$. Finally, one should remark that the scalar $a_\lambda$ does not depend on $\lambda$ because of the very first equality.
The lemma thus follows.
\end{proof} 

\begin{lemma}\label{remainingweight}
Keep the assumptions of the previous lemma and take $v_\gamma$ such that $v_\gamma^{\lambda_i}=0$ for all $\lambda_i$ orthogonal to $\beta$.
Suppose that $X_\nu v_\gamma\in\mathfrak g.v_{\underline\lambda} \setminus\{0\}$ for some simple root $\nu$ and there exists $\delta$ in the support of $\gamma$ such that $\gamma-\delta\not\in\Phi$.
Then  $(\lambda,\delta)\neq 0$  implies $(\lambda ,\beta)\neq 0$ for every highest weight $\lambda$ of $V$.
\end{lemma}

\begin{proof}
Note that $\gamma-\nu$ has to be a root hence the roots $\delta$ and $\nu$ are different.
It follows that $\delta$ lies in the support of $\gamma-\nu$ and in turn $v_\gamma^{\lambda_i}\neq 0$ for every $\lambda_i$ non-orthogonal to $\delta$.
The lemma thus follows from the choice of $v_\gamma$.
\end{proof}

 Let us now proceed to the proof of Proposition~\ref{supportofGamma}.

\begin{proof}
We proceed by contradiction: suppose $(\gamma,\delta)\leq 0$ and $(\lambda,\delta)\neq 0$ for some $\lambda\in\Delta$.

Note that if $(\gamma,\delta) \neq 0$ then the second assumption is automatically satisfied.
Indeed,  since $(\alpha,\delta)=0$ we have $(\gamma+\alpha,\delta)\neq 0$ and in turn, 
there exists $\lambda\in \Delta$ non-orthogonal to $\delta$ because  $\gamma\in\mathbb Z \Delta$ or $\gamma+\alpha\in\mathbb Z \Delta$.

Let $\delta'\in\supp\gamma$ be adjacent to $\delta$; such a root exists thanks to Lemma~\ref{weightshape}-(2).

Assume first there exists $\beta\in S$, $\beta\neq \alpha$ such that $(\gamma,\beta)>0$.
Note that $\beta\neq \delta$ by assumption.
For such a fixed $\beta$, we apply Lemma~\ref{remainingweight} and we get: $(\lambda,\delta)(\lambda,\beta)\neq 0$.
Further we take $v_\gamma$ as in loc. cit..
Thanks to Lemma~\ref{properties}-(4), we obtain that $\delta\sim\beta$ and in particular $(\beta,\delta)=0$.
It follows that the roots $\beta$, $\delta$ and $\delta'$ are pairwise distinct.
Since $\beta+\delta\in\Sigma$, it is in the $\mathbb Z$-span of $\Delta$ by Lemma~\ref{enlargeddefinitionofweights} hence every simple root adjacent to $\delta$ or $\beta$ does not belong to $S^p$.

Let $\nu\in\Phi$ be positive, different from $\alpha$ and such that $X_\nu v_\gamma\neq 0$. Such a root exists thanks to Lemma~\ref{non-highest} and by assumption, $X_\nu v_\gamma=X_{-\gamma+\nu}v_{\underline \lambda}$ (up to a scalar).
Further, $\supp(\gamma-\nu)\setminus \{\beta,\delta\}\subset S^p$. This inclusion may be obtained as follows. Let $\eta\in\supp(\gamma-\nu)\setminus \{\beta,\delta\}$.
If $\eta\not\in S^p$ then there exists $\lambda '\in \Delta$ such that $(\lambda ',\eta)\neq 0$.
The $\lambda'$-component of $X_\nu v_\gamma$ is thus not trivial and in turn $(\lambda ', \beta+\delta)\neq 0$ by the choice of $v_\gamma$.
We thus conclude by invoking Lemma~\ref{obviousppties}-(2).

Note that $\nu\neq\delta$ since $\gamma-\nu$ is a root but $\gamma-\delta$ is not.
Recall also that $\beta\neq \delta$.
Therefore, if $\nu\in S$ then the support of $\gamma-\nu$ contains $\delta$ and at least one the roots $\beta,\delta'$.
But since $\gamma-\nu\in\Phi$  and $(\beta,\delta)=0$, the support of $\gamma-\nu$ has to contain a simple root adjacent to $\delta$ and which is not in $S^p$ as already remarked.
This yields a contradiction with $\supp(\gamma-\nu)\setminus \{\beta,\delta\}\subset S^p$.

It follows that $X_\nu v_\gamma=0$ for every simple root $\nu\neq \alpha$ and in turn $\alpha$ belongs to the support of $\gamma$.
Note that $\alpha$ is different from $\beta,\delta$ and $\delta '$.
Furthermore, $X_{\alpha+\alpha '}v_\gamma\neq 0$ for some simple root $\alpha '$ since as already noticed in the proof of Lemma~\ref{non-highest}, the weight $\gamma$ can not be a multiple of $\alpha$.
Considering the support of $\nu-\alpha-\alpha '$ and arguing as in the previous paragraph, we get again a contradiction.

Assume now that $(\gamma,\beta)\leq 0$ for every simple root $\beta\neq\alpha$.
If $X_\nu v_\gamma\neq 0$ for a simple root $\nu\neq \alpha$ then $\gamma-\nu$ is a root and so is $\gamma$ since $(\gamma-\nu,\nu)<0$.
The weight $\gamma$ being a positive but non-simple root, we must have $(\gamma,\alpha)>0$ and also $(\gamma,\alpha+\alpha')>0$ for some simple root $\alpha'$ adjacent to $\alpha$.
If $X_\nu v_\gamma=0$  for every simple root $\nu\neq\alpha$ then (as already noticed) $X_{\alpha+\alpha''}v_\gamma\neq 0$ for some simple root $\alpha ''$ adjacent to $\alpha$ and in turn $\gamma-\alpha-\alpha''$ is a root.
If $\gamma\not\in\Phi$ then obviously $(\gamma,\alpha+\alpha'')>0$ otherwise as just remarked this inequality holds also, possibly with another simple root adjacent to $\alpha$.
In any case, we can thus apply Lemma~\ref{remainingweight} and we get: $(\lambda,\delta)\cdot(\lambda,\alpha+\alpha')\neq 0$ for  some simple root $\alpha'$ adjacent to $\alpha$.
Note that the roots $\alpha$, $\alpha'$  and $\delta$ are distinct and they do not belong to $S^p$ (by assumption and similar arguments used in the first case).
If further all the roots $\alpha$, $\alpha'$ and  $\delta$ and $\delta'$ are distinct, we get a contradiction by means of Lemma~\ref{additionalppty}.
We are thus left with $\gamma$ of support consisting only of the roots $\alpha$, $\delta$ and $\delta'$ (hence $\alpha'=\delta'$).
Straightforward considerations yield the desired contradiction.

\end{proof}

\subsection{Computations in degree $0$}

From now on, the spherical system $\mathscr{S}$ is assumed to be spherically closed and $V=V(\Delta)$.

Let $\Sigma(\Delta)$ denote the set of $T_\ad$-weights of $\left(V/\tilde{\mathfrak g}.v_{\Delta}\right)^{\tilde G_{v_{\Delta}}}$.

\begin{proposition}~\label{descriptionofweights}
\smallbreak\noindent{\rm(i)}\enspace
$\Sigma(\Delta)\subset \mathbb Z\Delta$.
\smallbreak\noindent{\rm(ii)}\enspace
$
\Sigma\subset\Sigma(\Delta)\subset\Sigma\cup \left\{\alpha+\alpha':( \alpha,\alpha')\neq 0	\mbox{ and } \alpha,\alpha'\in S\cap \Sigma\right\}.
$
\smallbreak\noindent{\rm(iii)}\enspace
The $T_\ad$-module $\left(V/\tilde{\mathfrak g}.v_{\Delta}\right)^{\tilde G_{v_{\Delta}}}$ is multiplicity-free.
\end{proposition}

\begin{remark}~\label{remark}
This proposition does not hold in general.
\smallbreak\noindent{\rm(i)}\enspace
Let us drop the requirement of being spherically closed for the spherical system.
Take for instance the spherical system with $\Sigma$ being given only by
the loose spherical root $\gamma=\alpha_1+\ldots+\alpha_r$ of type $\mathsf B_r$.
Then an easy computation shows that $2\gamma$ belongs to the related set $\Sigma(\Delta)$ but $\gamma$ itself does not.
\smallbreak\noindent{\rm(ii)}\enspace
Consider the spherical system $(\emptyset, \{\alpha_1+\alpha_2,\alpha_3+\alpha_4\},\emptyset)$ with $G$ of type $\mathsf A_4$.
If  we regard $V$ just as a $G$-module then the $T_\ad$-weights of $(V/\mathfrak g.v_\Delta)^{G_{v_\Delta}}$ are $\alpha_1+\alpha_2,\alpha_2+\alpha_3$ and $\alpha_3+\alpha_4$.
But the $\Sigma(\Delta)$ does not contain $(\alpha_2+\alpha_3,0)$, the latter not being in the integral span of the $\lambda_D$'s.
\end{remark}

Before proceeding to the proof of the preceding proposition, let us state and prove a consequence of this latter statement.

\begin{proposition}\label{consequence-descriptionofweights}
Let $\gamma=\alpha+\alpha'\in \Sigma(\Delta)\setminus\Sigma$ and $[v_\gamma]$ be the corresponding $T_\ad$-weight vector.
Then there exists a representative $v_\gamma$ of $[v_\gamma]$ such that $v_\gamma^\lambda\neq 0$ if and only if $(\lambda,\alpha)\neq 0$.
\end{proposition}

\begin{proof}
 Note that $\gamma$ is a positive root since $\alpha,\alpha'$ are adjacent (distinct) simple roots by Proposition~\ref{descriptionofweights}.

Take $v_\gamma\in \oplus V(\lambda_D)_{\lambda_D-\gamma}$.
To be definite, let  $X_{\alpha'}v_\gamma\neq 0$. Since $X_{\alpha'}v_\gamma\in\mathfrak g.v_\Delta$, the components $v_\gamma^{\lambda_D}$ are not trivial for $\lambda_D$ non-orthogonal to $\alpha$. 
By Proposition~\ref{descriptionofweights}, $\alpha,\alpha'\in \Sigma$.
Thanks to Lemma~\ref{properties}, we know that there are three of four $\lambda_D$ non-orthogonal to $\gamma$.

Assume first that there are four weights $\lambda_D$ non-orthogonal to $\gamma$.
Then $(\lambda_D,\alpha)(\lambda_D,\alpha')=0$ for all $D$ by Lemma~\ref{properties}
and, in turn, $v_\gamma^{\lambda_D}=X_{-\gamma}v_{\lambda_D}$ (up to a scalar).
Since $X_{\alpha}v_\gamma\in\mathfrak g.v_\Delta$ and $X_{\alpha'}v_\gamma\in\mathfrak g.v_\Delta$, there are either two of four non-trivial components of $v_\gamma$. An appropriate change of representative thus yields the desired representative of $[v_\gamma]$.

Assume now that there are three weights $\lambda_D$ non-orthogonal to $\gamma$ and $v_\gamma^{\lambda_D}\neq 0$ for all these weights $\lambda_D$.
Invoking Lemma~\ref{properties} again, we obtain that two of these weights $\lambda$ are such that $(\lambda,\alpha).(\lambda,\alpha')=0$ and in turn,
$v_\gamma^\lambda= X_{-\gamma}v_\lambda$ (up to a scalar) hence by an appropriate change of representative of $[v_\gamma]$, we can conclude our investigation.

\end{proof}

\paragraph{Proof of Proposition~\ref{descriptionofweights}.}

The first assertion of the proposition is obvious.
 Let us thus proceed to the proof of the two last assertions together.

Start with $\gamma\in\Sigma$.
If $\gamma\not\in S$ then we can apply Lemma~\ref{saturatedcase} to $\gamma$ since $\gamma$ is not a loose  spherical root by assumption.
Then, in coordination with Lemma~\ref{restriction}, we get $\gamma\in \Sigma(\Delta)$
as well as the multiplicity freeness for the $T_\ad$-weightvectors with weight in $\Sigma\setminus S$.
If $\gamma\in S$ then there exist  two weights $\lambda_D$, say $\lambda_\alpha^-$and $\lambda_\alpha^+$, non-orthogonal to $\alpha$ by Lemma~\ref{obviousppties}-(ii).
 It follows obviously that $\alpha\in \Sigma(\Delta)$; the (unique) corresponding $T_\ad$-weightvector is given by $[X_{-\alpha}v_{\lambda_\alpha^+}]=[X_{-\alpha}v_{\lambda_{\alpha}^-}]$.
This proves that $\Sigma\subset \Sigma(\Delta)$.

Conversely, if $\alpha\in S$ is a weight in $\Sigma(\Delta)$ then again by Lemma~\ref{obviousppties}-(ii), $\alpha\in \Sigma$.
Let $\Sigma(\Delta)'$ be the set of weights of $\Sigma(\Delta)$ whose support does not contain any element of $\Sigma$.
Let $\supp\Sigma(\Delta)'$ denote the union of the $\supp\gamma$'s with $\gamma\in\Sigma(\Delta)'$ and consider the related Levi subgroup $L$ of $G$ and the $L$-module as in Lemma~\ref{restriction}. 
Then by Remark~\ref{remarksaturated}-(i), we fall in the setting of Lemma~\ref{rootsupport}.
By Lemma~\ref{weightvector}-(ii) and Lemma~\ref{lindependencesphroots}, it follows that the elements of $\Sigma(\Delta)'$ are linearly independent characters of $T$. 
Moreover, if $\gamma\in \Sigma(\Delta)'$ then $(\gamma,0)\in \mathbb Z\Delta$ by (i).
By Lemma~\ref{enlargeddefinitionofweights} (and its proof) along with the just proved linear independence of the elements of $\Sigma(\Delta)'$ (which include that of $\Sigma$ as already shown), we must have $\gamma \in \Sigma$.

Let $\gamma\in \Sigma(\Delta)\setminus S$ whose support intersects $\Sigma$.
Further, thanks to the lemmas stated right below, 
there exist adjacent simple roots, say $\alpha$ and $\alpha'$ such that $\gamma$, $\gamma-\alpha$, $\gamma-\alpha'$ are all roots of $G$.
By simple considerations on roots, we get that $\gamma$ equals $\alpha+\alpha'$ or $\gamma$ is a spherical root of type $\mathsf C_r$. 
Invoking again Lemma~\ref{enlargeddefinitionofweights}, we obtain also that $\gamma\in\mathbb Z\Sigma$.
Finally, from Propositon~\ref{supportofGamma} along with the compatiblity condition shared by the elements of $\Sigma$, it follows that $\gamma\in\Sigma$ or $\alpha'\in \Sigma$.

In the following lemmas,  $\gamma$ and $\alpha$ are distinct weights in $\Sigma(\Delta)$ with $\alpha$ lying in the support of $\gamma$.
Recall that $\alpha\in\Sigma$ as already shown in the proof above. Let $\lambda^+_\alpha$ and $\lambda^-_\alpha$ denote the dominant weights in $\Delta$ which are not orthogonal to $\alpha$.

\begin{lemma}\label{lemma3}
The character $\gamma-\alpha$ is a root of $G$.
\end{lemma}

\begin{proof}
Let us proceed by contradiction: suppose $\gamma-\alpha$ is not a root.
Since $X_\alpha v_\gamma\in\mathfrak g.v_{\Delta}$, the vector $X_\alpha v_\gamma$ is trivial.
Moreover, by Proposition~\ref{supportofGamma}, $(\gamma,\alpha)$ is strictly positive.
By Lemma~\ref{goodrepresentative}, the representative $v_\gamma$ can be taken in $V(\lambda_\alpha^+)\oplus V(\lambda_\alpha^-)$.
Since the vector $v_\gamma$ can not be a highest weight vector of $V$,
there exists $\delta$ in the support of $\gamma$ such that the vector $X_\delta v_\gamma\neq 0$.
It follows that $\gamma-\delta\in \Phi$.
By similar arguments as those used in the proof of Proposition~\ref{supportofGamma}, we get that $\gamma$ belongs to the $\mathbb Z$-span of $\alpha$ and $\delta$.
Simple considerations thus show that $\gamma$ has to be a root: a contradiction with $\gamma-\alpha$ non-being a root and $(\gamma,\alpha)>0$.
\end{proof}

\begin{lemma}\label{lemma4}
If $\gamma\not\in\Phi$ then $X_\alpha v_\gamma\neq 0$ for a representative $v_\gamma$ of $[v_\gamma]$ in $V(\lambda_\alpha^+)\oplus V(\lambda_\alpha^-)$.
\end{lemma}

\begin{proof}
Thanks to Lemma~\ref{lemma3}, we know that $\gamma-\alpha$ is a root.
This together with $\gamma$ not being a root imply that
$(\gamma-\alpha,\alpha^\vee)\geq 0$ hence $(\gamma,\alpha)>0$.
By the same arguments as those used in the proof of Lemma~\ref{lemma3}, we get a contradiction whenever $X_\alpha v_\gamma=0$.
\end{proof}

\begin{lemma}\label{lemma5}
The supports of $\alpha$ and $\gamma-\alpha$ are not orthogonal.
\end{lemma}

\begin{proof}
Let us proceed by contradiction.
Then the weight vector $v_\gamma$ as in Lemma~\ref{lemma4} can be written as $X_{-\alpha}v_{\gamma-\alpha}$ where $v_{\gamma-\alpha}$ is a $T_\ad$-weight vector of weight $\gamma-\alpha$.
In particular, $v_\gamma$ can be taken in $V(\lambda^+_\alpha)\oplus V(\lambda^-_\alpha)$.
Furhter, $X_\alpha v_\gamma$ is not trivial hence in $\mathfrak g.v_{\Delta}\setminus\{0\}$.
Recall that $\alpha\in\Sigma\cap S$; by means of Lemma~\ref{properties}-(4), we get a contradiction.
\end{proof}

\begin{lemma}\label{lemma6}
There exists a simple root $\alpha'$ adjacent to $\alpha$ such that $\gamma-\alpha'\in\Phi$.
In particular, $\alpha'$ lies in the support of $\gamma$.
\end{lemma}

\begin{proof}
Note first that by the previous lemma, the support of $\gamma$ contains a simple root, say $\alpha'$, adjacent to $\alpha$.

Let us proceed by contradiction: suppose $\gamma-\alpha'$ is not a root.
Then $X_{\alpha'}v'_\gamma=0$ for any representative $v'_\gamma$ and $(\gamma,\alpha')>0$ by Proposition~\ref{supportofGamma}.
It follows that $\gamma$ is not a root and we can choose a representative $v'_\gamma$
such that its $\lambda_D$-components are trivial for every $\lambda_D$ orthogonal to $\alpha'$.
By Lemma~\ref{lemma4}, $X_\alpha v_\gamma\neq 0$ and in turn $X_\alpha v'_\gamma\neq 0$ since $\gamma\not\in\Phi$.
Therefore the support of $\gamma-\alpha$ does not contain the root $\alpha$ (Lemma~\ref{properties}-(4)).
This together with the fact that $\alpha'$ belongs to the support of $\gamma$ imply the inequality $(\gamma-\alpha,\alpha)<0$.
It follows that $\gamma$ is a root since so is $\gamma-\alpha$ by Lemma~\ref{lemma3}: a contradiction.
\end{proof}

\begin{lemma}
The weight $\gamma$ is a root of $G$.
\end{lemma}

\begin{proof}
We first claim the following.
Let $\alpha,\alpha'\in S$ be non-orthogonal and $\delta\in\Phi$.
If $ \delta+\alpha$ is not a root then neither is $\delta+\alpha-\alpha'$.

Apply this claim to $\delta:=\gamma-\alpha$ which is a root as previously proved.
We get that if $\gamma$ is not a root then neither is $\gamma-\alpha'$ for any simple root $\alpha'$ adjacent to $\alpha$.
This yields a contradiction with Lemma~\ref{lemma6}.
\end{proof}

\subsection{Computations in degree $1$}
In order to state the main theorem of this section, we need to introduce some additional notation.

First recall that $V=V(\Delta)$, the set $\Sigma(\Delta)$ is defined as the set of $T_\ad$-weights of $\left(V/\tilde{\mathfrak g}.v_{\Delta}\right)^{\tilde G_{v_{\Delta}}}$ and that
$S^p$ denotes the set of simple roots of $G$ orthogonal to every $\lambda_D$ in $\Delta$.

Given $\gamma\in \Sigma(\Delta)$, we let $[v_\gamma]\in\left(V/\tilde{\mathfrak g}.v_{\Delta}\right)^{\tilde G_{v_{\Delta}}}$ be the $T_\ad$-weight vector 
of weight $\gamma$ (Proposition~\ref{descriptionofweights}).
Further, we choose a representative $v_\gamma$ of $[v_\gamma]$ in $\oplus_{\lambda\in\Delta} V(\lambda)_{\lambda-\gamma}$.
In case $\gamma\in S$, we  consider the two possible such representatives of $[v_\gamma]$:
$v_\gamma^+=X_{-\gamma}v_{\lambda_\gamma^+}$ and $v_\gamma^-=X_{-\gamma}v_{\lambda_\gamma^-}$ where $\lambda_\gamma^\pm$ are the dominant weights in $\Delta$ which are not orthogonal to $\gamma$.

For $\alpha\in S$, let $s_\alpha$ denote the reflection of the Weyl group of $(G,T)$ associated to $\alpha$.

Given $\alpha,\gamma\in S$ with $\gamma\in \Sigma(\Delta)$, we set
$$
v_{\alpha*\gamma}^\pm=\left\{
\begin{array}{ll}
0\quad\mbox{if $v_{s_\alpha(\lambda_\gamma^\pm-\gamma)}=X_{-\alpha}X_{-\gamma}v_{\lambda_\gamma^\pm}$ with $(\alpha,\gamma)=0$} \\
v_{s_\alpha(\lambda_\gamma^\pm-\gamma)}\quad\mbox{otherwise}
\end{array}
\right. .
$$

\begin{lemma}\label{choiceofrepresentative}
Let $\gamma\in\Sigma\setminus S$ and suppose there are two distinct (up to scalar) representatives of $[v_\gamma]$ in $\oplus V(\lambda)_{\lambda-\gamma}$.
If further $X_{-\alpha}v_\gamma\neq 0$ for such a (hence every) representative of $[v_\gamma]$
then there exists $\alpha'\in S$ and a unique representative of $[v_\gamma]$ in
$\oplus V(\lambda)_{\lambda-\gamma}$ such that $(\alpha',\alpha)<0$ and $X_{\alpha'}v_\gamma\neq 0$.
\end{lemma}

\begin{proof}
This follows essentially from Lemma~\ref{saturatedcase}.
\end{proof}

\begin{definition}\label{Kostantvector}
For $\alpha\in S$ and $\gamma\in\Sigma(\Delta)\setminus S$, we set
$$
v_{\alpha*\gamma}=\left\{
\begin{array}{rl}
v_\gamma\quad\mbox{ if $\gamma\not\in\Sigma$}\\
X_{-\alpha}^r v_\gamma\quad\mbox{ if $\gamma\in\Sigma$}
\end{array}
\right. .
$$
In the latter, we let  $r=r(\alpha,\gamma)$ be maximal such that $X_{-\alpha}^r v_\gamma\neq 0$ and in case $X_{-\alpha}v_\gamma\neq 0$ we choose $v_\gamma$ as in Lemma~\ref{choiceofrepresentative}.

We denote the $T_\ad$-weight of $v_{\alpha*\gamma}$ and $v_{\alpha*\gamma}^\pm$ by $\alpha*\gamma$.
\end{definition}

We consider in the remainder, the $T_\ad$-action on $\tilde{\mathfrak g}^*\otimes V/\tilde{\mathfrak g}.v_{\Delta}$ given by the normalized action on
$\tilde{\mathfrak g}^*$ and on $V/\tilde{\mathfrak g}.v_{\Delta}$.
This yields in turn a $T_\ad$-action on the $\tilde G_{v_{\Delta}}$-module $H^1(\tilde G_{v_{\Delta}},V/\tilde{\mathfrak g}.v_{\Delta})$ via (\textsl{see}~\cite{Ho})
$$
H^1(\tilde G_{v_{\Delta}},V/\tilde{\mathfrak g}.v_{\Delta})=
\left(H^1(\tilde{\mathfrak g}_{v_{\Delta}},V/\tilde{\mathfrak g}.v_{\Delta})\right)^{\tilde{G}_{v_{\Delta}}/\tilde{G}_{v_{\Delta}}^\circ}.
$$

For a given $\alpha\in S$, let $\tilde{\alpha}$ denote the longest root of $G$ contained in the connected component of $\alpha$ in $\Phi$.

\begin{theorem}\label{LiealaKostant}
We have an isomorphism of $T_\ad$-modules

\begin{eqnarray*}
H^1(\tilde{\mathfrak g}_{v_{\Delta}},V/\tilde{\mathfrak g}.v_{\Delta})^{\tilde G_{v_\Delta}}
& \simeq &\bigoplus_{\substack{\gamma\in\Sigma(\Delta)\cap S\\ \alpha\in S\setminus S^p\\ \alpha*\gamma-\tilde\alpha+\alpha\in\mathbb Z\Delta}} k X_\alpha^*\otimes
[v_{\alpha*\gamma}^+]+k X_\alpha^*\otimes [v_{\alpha*\gamma}^-]\\
&        & \bigoplus _{\substack{\gamma\in\Sigma(\Delta)\setminus S\\ \alpha\in S\setminus S^p\\ \alpha*\gamma-\tilde\alpha+\alpha\in\mathbb Z\Delta}} k X_\alpha^*\otimes [v_{\alpha*\gamma}].
\end{eqnarray*}
\end{theorem}

\begin{remark}
The  vector $X_\beta^*\otimes [v_{\alpha*\gamma}^{(\pm)}]$ of $\tilde{\mathfrak g}^*\otimes V/\tilde{\mathfrak g}.{v_\Delta}$ has to be fixed by $\tilde G_{v_\Delta}$ whence the condition stated above on its $T_\ad$-weight $\alpha*\gamma-\tilde\alpha+\alpha$.
Since $\gamma$ is in the integral span of $\Delta$, this yields in fact a condition just on $\tilde\alpha$ and $\alpha$.
\end {remark}

The proof of the above theorem requires the following proposition.

\begin{proposition}\label{vanishingconditions}
Let $\varphi\in H^1(\tilde{\mathfrak g}_{v_{\Delta}},V/\tilde{\mathfrak g}.{v_{\Delta}})$ 
be a non-zero $T_\ad$-weight vector.
Then
$$
X_\beta\varphi(X_\alpha)=0
$$
for every simple root $\alpha$ and every root $\beta\neq \alpha$ of the isotropy Lie algebra $\tilde{\mathfrak g}_{v_{\Delta}}$.
\end{proposition}

\begin{remark}
The vanishing condition fulfilled by the vector $\varphi(X_\alpha)$ is that stated in Proposition~\ref{supportofGamma}.
\end{remark}

\paragraph{\textbf{Proof of Theorem~\ref{LiealaKostant}}}
Let $\varphi\in\left(H^1(\tilde{\mathfrak g}_{v_{\Delta}},V/\tilde{\mathfrak g}.v_{\Delta})\right)^{\tilde G_{v_{\Delta}}/\tilde G_{v_{\Delta}}^\circ}$ be a $T_\ad$-weight vector.
Let $\gamma$ be the $T_\ad$-weight of $\varphi(X_\alpha)$ and $v_\gamma$ denote a representative of $\varphi(X_\alpha)$ in $\oplus _\lambda V(\lambda)_{\lambda-\gamma}$.

Remark that when $\alpha$ does not belong to the support of $\gamma$ then by Proposition~\ref{vanishingconditions}, $\gamma$ lies in $\Sigma(\Delta)$.
We shall thus assume in the remainder of the proof that $\alpha$ does belong to the support of the $T_\ad$-weight $\gamma$
and that $v_\gamma$ is not equal to $X_{-\alpha}^r v_{\lambda_i}$ - in which case the proposition is obvious.
We shall proceed along the type of the support of $\gamma$.
Let us work out a few cases in detail.
The main ingredients of the proof are Proposition~\ref{supportofGamma} and ~\ref{vanishingconditions} along with the properties enjoyed by the dominant weights in $\Delta$ (\textsl{see} section~\ref{propertiesofthemonoid}).
As a consequence of Proposition~\ref{vanishingconditions}, the weight $\gamma$ can be written as a sum of two positive roots, say $\beta_1$ and $\beta_2$.

Consider first the case where the supports of the roots $\beta_1$ and $\beta_2$ are orthogonal.
Thanks to Proposition~\ref{supportofGamma}, the roots $\beta_1$ and $\beta_2$ have to be simple.
In virtue of Lemma~\ref{properties}, there is a single dominant weight, say $\lambda$,
which is neither orthogonal to $\beta_1$ nor to $\beta_2$.
Thanks to Proposition~\ref{vanishingconditions}, $\gamma\in\Sigma(\Delta)$.

Suppose now that the support of $\gamma$ is of type $\mathsf A_n$.
If $\gamma$ is not a root,
Proposition~\ref{supportofGamma} and ~\ref{vanishingconditions} yield:  $\gamma=\alpha_{i-1}+2\alpha_i+\alpha_{i+1}$ with $\alpha=\alpha_i$
and all the dominant weights $\lambda_k$ are orthogonal to both $\alpha_{i-1}$ and $\alpha_{i+1}$.
Clearly, we thus have: $[v_\gamma]\in\left(V/\mathfrak g.{v_{\Delta}}\right)^{G_{v_{\Delta}}}$.
If $\gamma$ is now a root then one gets: $\gamma=\alpha_i+\ldots+\alpha_j$ and $\alpha=\alpha_i$ (or $\alpha_j$) by Proposition~\ref{supportofGamma}.
Since $(\gamma,\alpha_j)>0$, applying the above remark to $\alpha_j$, we get that either $\gamma$ or $\gamma-\alpha$ belongs to $\Sigma(\Delta)$.

In case of type $\mathsf B_n$, we obtain similarly as before that  $\gamma=\alpha_i+\ldots+\alpha_n$ or $\gamma=2(\alpha_i+\ldots+\alpha_n)$ whenever $\alpha$ lies in the support of $\gamma$. Note that the same arguments as for the case of a root $\gamma$ of type $\mathsf A$ can be applied.
Suppose thus $\gamma$ is the weight $2(\alpha_i+\ldots+\alpha_n)$.
Then by the above remark along with Proposition~\ref{supportofGamma} and \ref{vanishingconditions} we get: $\alpha=\alpha_i$.
Moreover all simple roots, except $\alpha_i$ and $\alpha_{i+1}$, in the support of $\gamma$ are orthogonal to the weights in $\Delta$.
Remark that $\alpha_{i+1}$ can not be orthogonal to $S^p$ by Whitehead lemma.
From Lemma~\ref{properties}, we deduce that the fundamental weight attached to $\alpha_i$ (resp. $\alpha_{i+1}$) is the unique weight in $\Delta$ non-orthogonal to $\alpha_i$ (resp. $\alpha_{i+1}$).
It follows that $\gamma-2\alpha\in\Sigma(\Delta)$.

The other types can be worked out similarly.

\subsubsection{\textbf{Proof of Proposition~\ref{vanishingconditions}}}

For each $\alpha\in S$, let $s_\alpha$ denote the associated simple reflection in the Weyl group of $(G,T)$.

\begin{theorem}[\cite{Ko}]~\label{Kostant}
$$
H^1(\mathfrak g_{v_{\lambda}},V(\lambda))=\oplus_{\alpha} k X_\alpha^*\otimes v_{s_\alpha\lambda}\quad\mbox{ as $T$-modules}
$$
where $v_{s_\alpha\lambda}$ is a weight vector in $V(\lambda)$ of weight $s_\alpha\lambda$ and $\alpha$ is a simple root non-orthogonal to $\lambda$.
\end{theorem}

Let $\varphi$ be a non-zero $T_\ad$-weight vector in $H^1(\mathfrak g_{v_{\Delta}},V/\mathfrak g._{v_{\Delta}})^{G_{v_{\Delta}}/G_{v_{\Delta}}^\circ}$.
Then there exist $\alpha\in S$ and a $T_\ad$-weight vector $[v_\gamma]$ in $V/\mathfrak g.{v_{\Delta}}$
such that $\varphi(X_\alpha)=[v_\gamma]\neq 0$ and
one can write $\varphi$ as
\begin{equation}\label{harmonic}
\varphi=\sum_{\beta+\nu=\alpha+\gamma} X_\beta^*\otimes [v_\nu].
\end{equation}
Further, note that the $T_\ad$-weight of $\varphi$ is in $\mathbb Z\Delta$.

Consider the short exact sequence of $\mathfrak g_{v_{\Delta}}$-modules
\begin{equation}\label{ses}
0\longrightarrow \mathfrak g.v_{\Delta}\longrightarrow V\longrightarrow V/\mathfrak g.v_{\Delta}\longrightarrow 0
\end{equation}
and the associated long exact sequence in cohomology.

In order to prove Proposition~\ref{vanishingconditions} we shall
study separately the following situations regarding the positive roots $\beta$ whose support is contained in that of $\gamma$:
$(\gamma,\beta)<0$; when $\gamma-\beta$ is a root, we work out first the case when $\alpha$ and the support $\supp(\beta)$ of $\beta$ are orthogonal
and thereafter the case when they are not; finally, we consider the roots $\beta$ such that $\gamma-\beta$ is not a root
and $(\gamma,\beta)\geq 0$.
In each situation, we shall end up by means of general arguments with a list of few cases which can be easily worked out.

Before dealing with different situations, note that the following lemma holds in general.

\begin{lemma}\label{trivialcase}
Let $\beta\in$ be a positive root such that for any weight $\phi(\alpha,\beta)$ distinct to $\alpha$ in the integral span $\Phi(\alpha,\beta)$
of $\alpha$ and $\beta$,
$\gamma-\phi(\alpha,\beta)$ is neither a root nor trivial.
Then there exists a representative $v_\gamma\in V$ of $\varphi(X_\alpha)$ such that $X_\delta v_\gamma=0$ for every positive root $\delta\in\Phi(\alpha,\beta)$ distinct to $\alpha$, and for $\delta=-\alpha$.
\end{lemma}

\begin{proof}
Let $\mathfrak g(\alpha,\beta)$ be the Levi subalgebra of $\mathfrak g$ associated to the roots $\alpha$ and $\beta$.
Considering the aforementioned long exact sequence of cohomology restricted onto the Lie subalgebra
$$
\mathfrak g_{v_{\Delta}}(\alpha,\beta):=\mathfrak g_{v_{\Delta}}\cap \mathfrak g(\alpha,\beta),
$$
we get that $\varphi$ maps trivially in $H^2(\mathfrak g_{v_{\Delta}}(\alpha,\beta),V)$.
We thus conclude by means of Theorem~\ref{Kostant}.
\end{proof}

\subsubsection{}
Throughout this section, the support of $\beta$ is contained in that of $\gamma$ and $(\gamma,\beta)<0$.

The following lemma is obvious.

\begin{lemma}\label{gammabetanotaroot}
The weight $\gamma-\beta$ is not a root except if $\gamma=3\alpha_1+2\alpha_2$ is a root of type $\mathsf G_2$.
\end{lemma}

\begin{corollary}
If the roots $\alpha$ and $\beta$ span a root system of type $\mathsf A_1\times\mathsf A_1$ then $X_\beta\varphi(X_\alpha)=0$.
\end{corollary}

\begin{proof}
The statement follows from the two preceding lemmas.
\end{proof}

\begin{lemma}\label{gamma-alpha-beta_root}
Suppose $\gamma-\alpha-\beta$ is a root. Then the following assertions hold.
\begin{enumerate}
	\item The weight $\gamma-\alpha$ is a root.
	\item $(\beta,\alpha^\vee)=-1$.
  \item
The weight $\gamma$ is one of the following:
\smallbreak\noindent
{\rm(i)}\enspace
$\gamma=\beta+3\alpha+2(\alpha^++\ldots+\alpha_{n-1})+\alpha_n$ in type $\mathsf{C_n}$;
\smallbreak\noindent
{\rm(ii)}\enspace
$\gamma=\ldots+\alpha^-+2\alpha+\beta$ with $\beta=\alpha_n$ in type $\mathsf{B_n}$;
\smallbreak\noindent
{\rm(iii)}\enspace
$\gamma=\ldots+\beta^-+2\beta+2\alpha$ with $\alpha=\alpha_n$ in type $\mathsf{C_n}$;
\smallbreak\noindent
{\rm(iv)}\enspace
$\gamma=\beta+2\alpha$ with $\alpha=\alpha_n$ in type $\mathsf{C_n}$.
\end{enumerate}

\end{lemma}

\begin{proof}
The inequality $(\gamma-\alpha-\beta,\beta)<0$ yields the first assertion.

Since $\gamma-\alpha\in\Phi$, we have $(\gamma-\alpha,\beta^\vee)\geq 0$
and in turn $0\leq (\gamma-\alpha,\beta^\vee)<2$, \emph{i.e.} $(\gamma-\alpha,\beta^\vee)=0$ or $1$.
The lemma follows readily.
\end{proof}

Let us proceed now to the proof of Proposition~\ref{vanishingconditions} in the case under consideration.

Thanks to Lemma~\ref{trivialcase}, we can assume there exists $\phi(\alpha,\beta)$ in the integral span of $\alpha$ and $\beta$
such that $\gamma-\phi(\alpha,\beta)\in\Phi$.
Considering the long exact sequence of cohomology associated to~(\ref{ses}), $\phi(\alpha,\beta)$ is either a root or of shape $-\alpha+\delta+\delta'$ with $\delta$ and $\delta'$ being positive roots in the integral span of $\alpha$ and $\beta$.

Note that $(\gamma-a\alpha-b\beta,\beta^\vee)<0$ for any positive integers $a$ and $b$ with $a\leq b$.
Along with Lemma~\ref{gammabetanotaroot} (not $G_2$ type with $\gamma\in\Phi$), it follows that
the weight $\phi(\alpha,\beta)$ has to be $\alpha+\beta$ or $2\alpha+\beta$; the latter
weight occurs only in case $(\beta,\alpha^\vee)=-2$.

Suppose first that $\gamma-\alpha-\beta$ is a root.
A glance at the weight of $\varphi(X_\beta)$ (\emph{see} Lemma~\ref{gamma-alpha-beta_root}) shows that this vector is trivial: this weight should be equal to $\gamma+\alpha-\beta$ and
should fulfill the required property.
Further, since $2\alpha+\beta$ is not a root, $\varphi([X_\alpha,X_{\alpha+\beta}])=X_\alpha X_\beta\varphi(X_\alpha)-X_{\alpha+\beta}\varphi(X_\alpha)$ is trivial.
We shall prove that there is a representative of $\varphi(X_\alpha)$ in $V$ such that the corresponding representative of $\varphi([X_\alpha,X_{\alpha+\beta}])$ in $V$ is trivial in $V$; we thus obtain the proposition thanks to Lemma~\ref{trivialcase} and Lemma~\ref{gammabetanotaroot}.

Suppose $X_\beta\varphi(X_\alpha)$ is not trivial.
Then considering again the weights $\gamma$ listed in Lemma~\ref{gamma-alpha-beta_root},
we see that a representative $v_{\gamma-\beta}$ of $X_\beta\varphi(X_\alpha)$ can be taken to be in $V(\lambda)_{\lambda-\gamma+\beta}$
where $\lambda$ is not orthogonal to $\alpha$.
The support of $\gamma-\beta$ contains a simple root $\alpha'$ adjacent to $\alpha$ such that $(\gamma,\alpha')$ and $(\alpha,\alpha')$ differ.
From Lemma~\ref{properties}, we deduce that whatever $X_\beta\varphi(X_\alpha)$ is, the vector $X_\alpha v_{\gamma-\beta}$ does not lie in $\mathfrak g.v_{\Delta}\setminus\{0\}$.

Consequently, if $X_\beta X_\alpha v_\gamma$ is not trivial then the $\lambda$-component of $X_{\alpha+\beta}v_\gamma$ equals up to a scalar to $X_{-\gamma+\alpha+\beta}v_{\lambda}$ for $\lambda$ orthogonal to $\alpha$.
Since $(\gamma,\alpha+\beta)>0$, there exists a representative of $\varphi(X_\alpha)$ whose $\lambda$-component is trivial for every dominant weight $\lambda$ orthogonal to $\alpha+\beta$. 
Note that such a dominant weight $\lambda$ exists under the assumption that $\gamma$ is distinct to $\beta+2\alpha$.

Let now $\gamma-\alpha-\beta$ not be a root.
As mentioned above, the weight $\gamma-2\alpha-\beta$ has to be a root and so has $2\alpha+\beta$.
We proceed similarly as above while considering instead
$[v]=X_\alpha\varphi(X_{2\alpha+\beta})-X_{2\alpha+\beta}\varphi(X_\alpha)$ - which is obviously trivial because of the cocycle property.

Assume $\varphi(X_{2\alpha+\beta})$ is not trivial.
One may list the possible roots $\gamma-2\alpha-\beta$ with $2\alpha+\beta$ being also a root and $(\gamma,\beta)<0$.
A glance at the weight $\gamma'=\gamma-\alpha-\beta$ of $\varphi(X_{2\alpha+\beta})$
shows that the representative $v_{\gamma'}$ of $\varphi(X_\alpha)$
in $\oplus_\lambda V(\lambda)_{\lambda-\gamma'}$ projects trivially onto $V(\lambda)$ if $\lambda$ is orthogonal to $\alpha$.
For such a $v_{\gamma'}$, $X_\alpha v_{\gamma'}$  does not lie in $\mathfrak g.v_{\Delta}\setminus\{0\}$.
If $X_\alpha v_{\gamma'}-X_{2\alpha+\beta}v_\gamma$ is not trivial in $V$ for some representative $v_\gamma$ of $\varphi(X_\alpha)$
then any $\lambda$-component of  $X_{2\alpha+\beta}v_\gamma$ has to be non-trivial whenever $\lambda$ is orthogonal to $\alpha$.
Further, $X_{2\alpha+\beta}. v_\gamma^\lambda=X_{-\gamma+2\alpha+\beta}v_\lambda$.
Since $(\lambda-\gamma,2\alpha+\beta)<0$ for $\lambda$ orthogonal to both $\alpha$ and $\beta$ (existence), there exists a representative of $\varphi(X_\alpha)$ whose $\lambda$-components
are trivial for $\lambda$ orthogonal to $\alpha$ and $\beta$.
It follows that the corresponding representative of $[v]$ is trivial in $V$ whence the proposition.

\subsubsection{$\gamma-\beta$ is a root with $\alpha$ and $\supp(\beta)$ being orthogonal}\label{mainsubsection}

First observe that $X_\beta\varphi(X_\alpha)$ is trivial whenever so is $\varphi(X_\beta)$ (thanks to the cocycle property). We shall thus suppose in the two following subsections that $\varphi(X_\beta)$ is not trivial;
let $\gamma'$ be its $T_\ad$-weight. Recall that $\gamma'=\gamma+\alpha-\beta$.

\subsubsection{\textbf{$(\gamma,\alpha)\leq 0$} with assumptions of~\ref{mainsubsection}}

\begin{lemma}\label{2ndtrivialcase}
Assume there exists $\delta\in\Phi$ positive such that
$\varphi$ restricted onto the Lie subalgebra $\mathfrak g(\alpha,\delta)$ associated to $\alpha$ and $\delta$ maps trivially onto $H^2(\mathfrak g(\alpha,\delta),V)$.
Then $(\gamma,\alpha)=0$ and there exists
a representative $v_\gamma\in V$ of $\varphi(X_\alpha)$ such that the $\lambda$-component of $v_\gamma$ is trivial
for every $\lambda$ non-orthogonal to $\alpha$.
\end{lemma}

\begin{proof}
The lemma follows readily from Kostant Theorem~\ref{Kostant} and the aforementioned long exact sequence.
\end{proof}

Let $v_{\gamma'}$ be a representative of $\varphi(X_\beta)$ in $\oplus V(\lambda)_{\lambda-\gamma'}$.
Suppose the assumptions of the lemma right above are satisfied and let $v_\gamma\in V$ be as in this lemma.
In particular, we have $X_{\pm\alpha}v_\gamma=0$ (in $V$).
Thanks to  the cocycle condition (applied to the roots $\alpha$ and $\beta$), we have $X_{\alpha}\varphi(X_{\beta})=X_\beta\varphi(X_\alpha)$
and in turn $X_{-\alpha}X_\alpha\varphi(X_\beta)=X_{-\alpha}X_\beta\varphi(X_\alpha)$.
Let $\lambda\in\Delta$ be such that $(\lambda,\gamma)\neq 0$ and $(\lambda,\alpha)=0$.
Note that such a weight $\lambda$ exists otherwise $v_\gamma$ will be $0$ thanks to the preceding lemma.
Since $(\gamma',\alpha)=(\gamma+\alpha-\beta,\alpha)=(\alpha-\beta,\alpha)>0$, the $\lambda$-component of the left hand side equals $\varphi(X_\beta)$ up to a scalar.
The right hand side equals $X_\beta X_{-\alpha}\varphi(X_\alpha)$ which is $0$.
It follows that $\varphi(X_\beta)$ has to be a trivial - which contradicts our assumption.
The proposition follows in the case under consideration.

Assume now that we are not in the setting of the lemma right above.
This implies in particular that $\alpha$ belongs to the support of $\gamma$.
Furthermore, at least one of the adjacent simple roots to $\beta$, say $\beta^-$, should belong also to the support of $\gamma$ (\textsl{see} Lemma~\ref{trivialcase}).

The following claim will be used in the following; it can be easily checked out by standard arguments.

\begin{claim}
If $\gamma-\beta$ and $\gamma-\beta^-$ are roots then so is $\gamma$.
\end{claim}

Assume first that $\gamma$ is not a root then thanks to this claim, neither $\gamma-\beta^-$ nor $\gamma-\alpha$ is a root.
The latter is due to the fact that $(\gamma-\alpha,\alpha)<0$ (recall that $(\gamma,\alpha)\leq 0$ by assumption).
Further by Lemma~\ref{trivialcase}, $\gamma-\alpha-\beta^-$ has to be a root hence $(\gamma-\alpha-\beta^-,\alpha)\geq 0$ and in turn $(\beta^-,\alpha^\vee)=-2$.
We then observe that no weight $\gamma$ falls in the case under study hence $\gamma$ has to be a root.

Assume thus now that $\gamma\in\Phi$.
Recall that $(\gamma,\alpha)\leq 0$, $\gamma-\beta$ is a root and $\gamma$ does not satisfy the conditions of the above lemma.
One obtains a few roots $\gamma$ and can conclude as before.

\subsubsection{\textbf{$(\gamma,\alpha)>0$} with assumptions of~\ref{mainsubsection}}

If $(\gamma',\beta)\leq 0$ then as proved in the preceding paragraph, $X_\alpha\varphi(X_\beta)$ is trivial and so is $X_\beta\varphi(X_\alpha)$ (by cocyclicity).
Let us thus assume that $(\gamma',\beta)$ is strictly positive \emph{i.e.} $(\gamma-\beta,\beta)>0$ since $(\alpha,\beta)=0$.
The weight $\gamma-\beta$ being a root, we have either $(\gamma-\beta,\beta^\vee)=1$ or $2$ whenever not of type $\mathsf G_2$.
One can thus list the very few possible roots $\gamma-\beta$.
Let us work out explicitly the type $\mathsf C_n$; we have either $\gamma-\beta=\alpha+\ldots+ 2(\beta+\ldots)+\alpha_n$
or $\gamma-\beta=\beta+\ldots+ 2(\alpha+\ldots)+\alpha_n$.

Consider the first possible weight $\gamma$.
Note that  $(\gamma',\alpha)\geq 3$ hence $(\lambda-\gamma',\alpha^\vee)<0$ for every $\lambda\in \Delta$ (Lemma~\ref{properties}-(1)).
Hence if $X_{-\alpha}\varphi(X_\alpha)=0$ then we can prove as before that $\varphi(X_\beta)=0$.
Otherwise, there exists a weight in $\Delta$ that is orthogonal  neither to $\alpha$ nor to $\beta$.
Recall that the $T_\ad$-weight of $\varphi$, that is $\gamma+\tilde{\alpha}-\alpha$, lies in the $\mathbb Z$-span of $\Delta$.
This together with the properties of the weights in $\Delta$ (\textsl{see} Section~\ref{propertiesofthemonoid}) imply:
\begin{claim}
There exists a weight in $\Delta$ which is orthogonal to $\alpha+\beta$ but non-orthogonal to $\gamma$.
\end{claim}

From this claim, it follows that the cocycle identity fufilled by $\varphi$ and for $\alpha$ and $\beta$ can be lifted up to $V$: $X_\beta v_\gamma -X_\alpha v_{\gamma'}=0$. 
Therefore $\varphi$ maps trivially in $H^2(\mathfrak g(\alpha;\beta), V))$ and as before we obtain $X_\beta\varphi(X_\alpha)=0$.

\subsubsection{\textbf{$\gamma-\beta$ is a root with $\alpha$ and $\supp(\beta)$ non-orthogonal}}

\begin{lemma}
We have $(\gamma,\alpha)\geq 0$ unless $\gamma=\alpha_{n-2}+\alpha_{n-1}+\alpha_n$ of type $\mathsf C_n$.
\end{lemma}

\begin{proof}
Let us proceed by contradiction.

Suppose first that $(\gamma-\beta,\alpha)$ is strictly positive.
Then $(\beta,\alpha^\vee)$ equals $-2$ or $-3$ and $\gamma-\beta-\alpha$ has to be a root.
In type $\mathsf B_n$, the simple root $\alpha_n$ has to be $\alpha$ itself and $(\gamma-\beta,\alpha)$ being strictly positive, it has to
be equal to $2$ and in turn $(\gamma,\alpha)=0$ -whence a contradiction.
Similarly, in type $\mathsf C_n$, we get as possibilities for $\gamma$ the weights $\gamma_1=\alpha_i+\ldots+2\alpha_{n-1}+\alpha_n$ with $i<n-1$ and $\beta=\alpha_n$ and $\gamma_2=\ldots+\alpha_{n-2}+\alpha_{n-1}+\alpha_n$.
Note that $\gamma_1-\phi(\alpha,\alpha_i)$ is not a root for any weight in the $\mathbb Z$-span of $\alpha$ and $\alpha_i$.
Together with Lemma~\ref{trivialcase}, this yields a contradiction with $(\gamma,\alpha)$ being strictly negative.
And similarly, we are left with $\gamma=\alpha_{n-2}+\alpha_{n-1}+\alpha_n$.
We handle the type $\mathsf F_4$ by analogous arguments.

Suppose now that $(\gamma-\beta,\alpha)$ is negative.
Assume further that the support of $\gamma$ contains a simple root $\delta$ which is orthogonal to $\alpha$.
Then by Lemma~\ref{2ndtrivialcase} again, the weight $\gamma-\delta$ has to be a root.

\begin{claim}
The weight $\gamma$ is a root and $\gamma=\beta+\delta$.
\end{claim}

Indeed, if $\alpha$ belongs to the support of $\gamma$ then
$\gamma$ is not a root and neither is $\gamma-\delta$.
Therefore the simple root $\alpha$ does not belong to the support of $\gamma$ and  the claim follows.

Finally assume that there is no root orthogonal to $\alpha$ in the support of $\gamma$.
It follows that $\alpha$ does belong to the support of $\gamma$ whenever $\gamma$ is distinct to $2\beta$.

\begin{claim}
The weight $\gamma-\beta$ is one of the roots $\alpha+\beta$, $2\beta+\alpha$, $2\alpha^-+\alpha$ or $\alpha^-+\alpha+\beta$.
\end{claim}

To obtain the above claim, we list the possible roots $\gamma-\beta$ such that there is no simple root $\delta$
orthogonal to $\alpha$  in the support of $\gamma$.

\end{proof}

We may assume without loss of generality that $\beta$ is a simple root not orthogonal to $\alpha$.
Therefore $\gamma-\alpha-\beta$ is a root since $(\gamma-\beta,\alpha)=(\gamma,\alpha)-(\alpha,\beta)>0$.

If $(\gamma',\beta)>0$ then $(\gamma-\beta,\beta^\vee)>-(\alpha,\beta^\vee)$ and in turn $(\gamma-\beta,\beta^\vee)\geq 2$
whence a contradiction with $\gamma-\alpha-\beta$ being a root.
It follows that $(\gamma',\beta)\leq 0$.
Then $(\gamma-\beta,\beta^\vee)\leq -(\alpha,\beta^\vee)$ and listing the possible weights, one may conclude as before taking into account the cases already worked out.

\subsubsection{\textbf{$\gamma-\beta$ is not a root and $(\gamma,\beta)\geq 0$}}

Remark that if $X_\beta\varphi(X_\alpha)$ is not trivial then (by Lemma~\ref{trivialcase}), the simple roots $\alpha$ and $\beta$ are not orthogonal
and in turn, one of the weights $\gamma-\alpha-\beta$ and $\gamma-2\alpha-\beta$ has to be a root.
Note that the latter may occur only in case $2\alpha+\beta$ is a root.

Suppose first that $\gamma-\alpha-\beta$ is a root.

\begin{claim}
$\varphi(X_\beta)$ is trivial.
\end{claim}

If $(\beta,\alpha^\vee)= -1$ then $(\gamma,\alpha)>0$
and $0=\varphi([X_\alpha,X_{\alpha+\beta}])=X_\beta X_\alpha\varphi(X_\alpha)$.
If there exists a representative $v_\gamma$ such that $X_\beta X_\alpha v_\gamma$ is trivial (in $V$) then the proposition is proved.
Let thus $X_\beta X_\alpha v_\gamma$ be non-trivial then it equals $X_{-\gamma+\alpha+\beta} v_{\Delta}$.
Note that whenever $\lambda$ is orthogonal to $\alpha$ , we have $(\lambda-\gamma+\alpha,\beta)<0$ in the case under study.
It follows that $X_\alpha v_\gamma=X_{-\beta}X_{-\gamma+\alpha+\beta}v_\lambda$ for $\lambda$ orthogonal to $\beta$ (recall that $X_\alpha v_\gamma$
may be assumed to be non-trivial otherwise the lemma is already proved).

\begin{claim}
There exists a dominant weight in $\Delta$ which is orthogonal to $\alpha+\beta$.
\end{claim}

Since $(\gamma,\alpha)>0$, we can deduce the existence of a representative $v_\gamma$
whose $\lambda$-components are trivial when $\lambda$ is orthogonal to $\alpha+\beta$.
The proposition follows by the same arguments as before.

Suppose now that $(\beta,\alpha^\vee)=-2$ (and $\gamma-2\alpha-\beta$ may be a root).
The possible weights can be explicitly listed.

\subsection{Application}\label{injectivity}

Retain the notation set up previously in this appendix and put

\begin{eqnarray*}
S^2 V/V(\Delta^2) & = & \oplus_{D\in\Delta} S^2 V(\lambda_D)/V(2\lambda_D) \\
                  &   & \oplus_{D\neq D'\in\Delta}V(\lambda_D)\otimes V(\lambda_{D'})/V(\lambda_D+\lambda_{D'}).
\end{eqnarray*}
A vector in any of the above direct summands is denoted by $v_D\cdot v_{D'}$.

Let $v_D$ denote the projection of $v\in V$ onto $V(\lambda_D)$ and consider the map of $T_\ad$-modules
\begin{eqnarray*}
f:&  V/\mathfrak g.v_{\Delta}&\rightarrow S^2 V/V(\Delta^2)\\
  & [v=\sum_{\Delta} v_D]&\mapsto [v\cdot v_{\lambda_D}]:=\sum_{D,D'}[v_D\cdot v_{\lambda_{D'}}].
\end{eqnarray*}
The referred $T_\ad$-module structure is induced by the normalized action on $V$.
We have obviously
\begin{lemma}
The map $f$ is injective.
\end{lemma}

\begin{proposition}\label{triviality}
The map induced by $f$
$$
H^1(f):H^1(\mathfrak g_{v_{\Delta}},V/\mathfrak g.v_{\Delta})\rightarrow H^1(\mathfrak g_{v_{\Delta}},S^2 V/V(\Delta^2))
$$
is injective.
\end{proposition}

\begin{proof}
Let $\varphi$ be a $T_\ad$-weight vector in $H^1(\mathfrak g_{v_{\Delta}},V/\mathfrak g.v_{\Delta})$.
By Proposition~\ref{LiealaKostant}, there exist $\alpha$ simple and $\gamma\in \Sigma(\Delta)$ such
that
$$
\varphi=X_\alpha^*\otimes [v_{s_\alpha*\gamma}].
$$

Let $[v_{s_\alpha*\gamma}\cdot v_{\lambda_D}]$ be non-trivial in $\mathop S^2V/V(\Delta^2)$; \textsl{see} the above lemma.
We shall prove that there is no $v\in\mathop S^2V/V(\Delta^2)$
such that $[v_{s_\alpha*\gamma}\cdot v_{\lambda_D}]=X_\alpha v$ in $\mathop{S^2}V/V(\Delta^2)$.
Note that there is no such $v$ whenever one of the following assertions holds:
\begin{equation}\label{1tobeproved}
\left[X_{-\alpha}\bigl(v_{s_\alpha*\gamma}\cdot v_{\lambda_D}\bigr)\right]=0\quad\mbox{ in }\mathop{S^2}V/V(\Delta^2).
\end{equation}

\begin{equation}~\label{2tobeproved}
X_\alpha^a v_{s_\alpha*\gamma}\neq 0\mbox{ in $V$}\quad\mbox{ for } a=(\lambda_D,\alpha^\vee).
\end{equation}

Let us first consider $X_{-\alpha}\bigl(v_{s_\alpha*\gamma}\cdot v_{\lambda_D}\bigr)$.
Note that by definition, we have: $X_{-\alpha}v_{s_\alpha*\gamma}=0$ in $V$.
We thus have
$$
X_{-\alpha}\bigl(v_{s_\alpha*\gamma}\cdot v_{\lambda_D}\bigr)=v_{s_\alpha*\gamma}\cdot X_{-\alpha}v_{\lambda_D}.
$$

Remark that if $(\lambda_D,\alpha)=0$, Assertion~(\ref{1tobeproved}) obviously holds.
The proposition thus follows from the next lemmas.
\end{proof}

\begin{lemma}\label{lemma1injectivity}
Let $v_{s_\alpha*\gamma}=X_{-\alpha}^r v_{\lambda_D}$ for some $\lambda_D$.
Then Assertion~(\ref{2tobeproved}) holds for a weight $\lambda_{D'}$ such that $[v_{s_\alpha*\gamma}\cdot v_{\lambda_{D'}}]\neq 0$.
\end{lemma}

\begin{proof}
Note that $r=(\lambda_D,\alpha^\vee)$.
Hence if $r>1$ then $[X_{-\alpha}^r v_{\lambda_D}\cdot v_{\lambda_D}]\neq 0$ and Assertion~(\ref{2tobeproved}) is clear whence the lemma with $D'=D$ itself.
If $r=1$ then necessarily $\alpha\in\Sigma$ and there exists $\lambda_{D'}\neq \lambda_D$ non-orthogonal to $\alpha$.
Thanks to Lemma~\ref{properties}, $(\lambda_{D'},\alpha^\vee)=1$ and Assertion~(\ref{2tobeproved}) holds with $\lambda_{D'}$.
\end{proof}

\begin{lemma}
If $(\gamma,\alpha)<0$ then Assertion~(\ref{2tobeproved}) holds for every $\lambda_D\in\Delta$.
\end{lemma}

\begin{proof}
Recall that $\gamma$ belongs to $\mathbb Z\Delta$ (Theorem~\ref{LiealaKostant}).
Together with Lemma~\ref{properties}, this implies that $(\lambda-\gamma,\alpha^\vee)\geq (\lambda',\alpha^\vee)$ for every $\lambda,\lambda'\in\Delta$.
The lemma follows readily.
\end{proof}

\begin{lemma}\label{lemma2injectivity}
Let $[v_{s_\alpha*\gamma}\cdot v_{\lambda_D}]\neq 0$ in $\mathop S^2V/V(\Delta^2)$ with $\gamma\in \Sigma(\Delta)$.
If $\lambda_D$ is orthogonal to $\gamma$  then Assertion~(\ref{2tobeproved}) holds.
\end{lemma}

\begin{proof}
Note first that the support of $\gamma$ does not contain $\alpha$.
Indeed $\lambda_D$ being orthogonal to $\gamma$ it can not be orthogonal to $\alpha$ otherwise $[v_{s_\alpha*\gamma}\cdot v_{\lambda_D}]$ will be $0$.
Hence $(\gamma,\alpha)\leq 0$ and further $X_\alpha v_\gamma=0$ in $V$.
By Proposition~\ref{vanishingconditions} together with Definition~\ref{Kostantvector}, we get that $v_{s_\alpha*\gamma}=X_{-\alpha}^r v_\gamma$ with $r=(\lambda-\gamma,\alpha^\vee)$ and $v_\gamma\in V(\lambda)$.
Further, since $[v_{s_\alpha*\gamma}.v_{\lambda_D}]\neq 0$ and $(\lambda_D,\gamma)=0$, we have:
$v_{s_\alpha*\gamma}\neq v_\gamma$.

The weight $\lambda$ being non-orthogonal to $\gamma$, it is different from $\lambda_D$.

Assume that $(\gamma,\alpha)=0$ then since $\alpha$ does not belong to the support of $\gamma$,
it has to be orthogonal to every simple root lying in the support of $\gamma$.
Let $\delta\in\supp\gamma$ be such that $X_\delta v_\gamma\neq 0$ (in $V$).
Then $X_\delta v_\gamma=X_{-\gamma+\delta}v_\lambda\in\mathfrak g.v_{\Delta}$.
Moreover, since $v_{s_\alpha*\gamma}\neq v_\gamma$, we have $(\lambda,\alpha)\neq 0$.
It follows from Lemma~\ref{obviousppties} that $\alpha\in\Sigma\cap S$ and $\gamma\in S$.
By Lemma~\ref{properties}-(4), we end up with a contradiction.
We deduce that $(\gamma,\alpha)<0$ - case worked out in the previous lemma.
\end{proof}

\begin{lemma}
Suppose $(\gamma,\alpha^\vee)>1$ with $\gamma\in \Sigma(\Delta)$.
Then $[v_{s_\alpha*\gamma}\cdot v_{\lambda_D}]$ and Assertion~(\ref{2tobeproved}) holds with $\lambda_D$ such that $v_\gamma\in V(\lambda_D)$.
\end{lemma}

\begin{proof}
If $\gamma=2\alpha$, we fall in the setting of Lemma~\ref{lemma1injectivity}.
Suppose thus that $\gamma\neq 2\alpha$.

In light of the description of $\Sigma(\Delta)$ (Proposition~\ref{descriptionofweights}), the weight $\gamma\in\Sigma(\Delta)$ under consideration is such that $(\gamma,\alpha^\vee)=2$ and $\gamma\in\Sigma$.
From the list of spherical roots together with Lemma~\ref{weightvector},
one can choose $v_\gamma$ in some $V(\lambda_D)$ of first type and $(\lambda_D,\alpha^\vee)=1$ necessarily.
\end{proof}

\begin{lemma}
Let $\gamma\in \Sigma(\Delta)$ and $(\gamma,\alpha^\vee)=1$ for some simple root $\alpha$.
Then Assertion~(\ref{2tobeproved}) holds.
\end{lemma}

\begin{proof}
By Proposition~\ref{descriptionofweights} and the table of spherical roots, $\gamma-\alpha$ is a root.
By Lemma~\ref{properties}, there exists $\lambda\in\Delta$ non-orthogonal to $\alpha$ and $(\lambda,\alpha^\vee)=1$.
Further $v_\gamma$ can be chosen in $V(\lambda)$.
If $X_\alpha v_\gamma=0$ in $V$, it follows from ~\cite{Js} and Lemma~\ref{weightvector} that $[v_\gamma]=[X_{-\gamma} v_{\lambda}]=[X_{-\gamma} v_{\lambda'}]$
for some $\lambda'\neq \lambda$ and such that $(\lambda',\gamma-\alpha)\neq 0$.
In particular $X_\alpha v_\gamma\neq 0$ in $V$ for $v_\gamma=X_{-\gamma}v_{\lambda'}$.
Then $v_\gamma$ can be chosen such that $v_\gamma\in V(\nu)$ with  $X_\alpha v_\gamma\neq 0$ in $V$ and $\nu=\lambda$ or $\lambda'$ as above.
Assertion~(\ref{2tobeproved}) thus holds with $\lambda_D=\lambda$.
\end {proof}

\begin{lemma}
Let $\alpha$ be a simple root not in $S^p$.
Suppose $(\gamma,\alpha)=0$ then one of the assertions ~(\ref{1tobeproved}) and (\ref{2tobeproved}) holds.
\end{lemma}

\begin{proof}
First assume that $\alpha$ does not belong to the support of $\gamma$.
Then $\alpha$ is orthogonal to every simple root in the support of $\gamma$.
It follows that $v_{s_\alpha*\gamma}^\lambda=v_\gamma^\lambda$ in $V$ if and only if $(\lambda,\alpha)=0$.
If the $\lambda$-component $v_\gamma^\lambda$ is such that $v_\gamma^\lambda\cdot v_\lambda\neq 0$ then Assertion~(\ref{1tobeproved}) holds whenever $(\lambda,\alpha)=0$.
If $v_\gamma^\lambda\cdot v_\lambda=0$ then $v_\gamma^\lambda=X_{-\gamma}v_\lambda$ and there exists $\lambda'\neq \lambda$ such that
$0\neq v_\gamma^\lambda\cdot v_{\lambda'}=v_\lambda\cdot X_{-\gamma}v_{\lambda'}$ hence Assertion~(\ref{1tobeproved}) whenever $(\lambda,\alpha)=0$.
Let now $(\lambda,\alpha)\neq 0$. Note that $X_\alpha v_\gamma=0$ in $V$ since $\alpha$ does not belong to the support of $\gamma$.
Then $v_{s_\alpha*\gamma}^\lambda= X_{-\alpha}^r v_\gamma^\lambda$ with $r=(\lambda-\gamma,\alpha^\vee)=(\lambda,\alpha^\vee)$.
Further $\gamma+\alpha$ is not a root therefore $v_{s_\alpha*\gamma}^\lambda\not\in\mathfrak g.v_{\lambda}$.
Assertion~(\ref{2tobeproved}) thus holds with $\lambda$ such that the $\lambda$-component of $v_\gamma\in \oplus_\lambda V(\lambda)_{\lambda-\gamma}$ is not trivial.

Assume now that $\alpha$ lies in the support of $\gamma$.
Then by Lemma~\ref{rootsupport}, $\gamma-\alpha$ has to be a root; the type $\mathsf F_4$ is easily ruled out.
More precisely $\gamma$ is a root of type $\mathsf B_n$ or $\mathsf C_n$.
Further in type $\mathsf B_n$, we can choose $v_\gamma= X_{-\gamma}v_\lambda$
whereas $v_\gamma\in V(\lambda)\setminus \mathfrak g.v_\lambda$ in type $\mathsf C_n$ along with $(\lambda,\alpha)=0$ in both cases.
In the first situation, $v_{s_\alpha*\gamma}=X_{-\gamma-\alpha}v_\lambda$ and there exists $\lambda'\neq \lambda$ non-orthogonal to $\gamma$.
In type $\mathsf B_n$, we then have $0\neq v_{s_\alpha*\gamma}\cdot v_{\lambda'}=X_{-\gamma-\alpha}v_{\lambda'}\cdot v_\lambda$ whence
Assertion~(\ref{1tobeproved}).
In type $\mathsf C_n$, Assertion~(\ref{1tobeproved}) holds with $\lambda'=\lambda$.
\end{proof}


\begin{thebibliography}{1000}

\bibitem[A]{A} Akhiezer, D., \textit{Equivariant completions of homogeneous algebraic varieties by homogeneous divisors}, Ann. Global Anal. Geom.  \textbf{1}  (1983), no. 1, 49--78.

\bibitem[AB]{AB} Alexeev, V. and Brion, M., \textit{Moduli of affine schemes with reductive group action}, J.\ Algebraic Geom., \textbf{14} (2005), no.\ 1, 83--117.

\bibitem[Bo]{Bo} Bourbaki, N., \'El\'ements de math\'ematique. Groupes et Alg\`ebres de Lie. Chapitre IV: Groupes de Coxeter et syst\`emes de Tits. Chapitre V: Groupes engendr\'es par des r\'eflexions. Chapitre VI: Syst\`emes de racines. Actualit\'es Scientifiques et Industrielles, No.\ 1337 \textit{Hermann, Paris,} 1968.

\bibitem[Bra]{Bra} Bravi, P., \textit{Wonderful varieties of type E}, Represent.\ Theory \textbf{11} (2007), 174--191.

\bibitem[BC1]{BC1} Bravi, P. and Cupit-Foutou S., \textit{Equivariant deformations of the affine multicone over a flag variety}, Adv.\
  Math.\ \textbf{217} (2008), 2800--2821.

\bibitem[BC2]{BC2} Bravi, P. and Cupit-Foutou S., \textit{Classification of strict wonderful varieties}, Ann. Institut Fourier, \textbf{60}, No. 2, (2010), 641--681.

\bibitem[BL]{BL} Bravi, P. and Luna D., \textsl{An introduction to wonderful varieties with many examples of type $\mathsf F_4$}, J. Algebra \textbf{329} (2011), 4--51.

\bibitem[BP]{BP} Bravi, P. and Pezzini, G., \textit{Wonderful varieties of type $D$}, Represent.\ Theory \textbf 9 (2005), 578--637.

\bibitem[BP11]{BP11} Bravi, P. and Pezzini, G., \textit{Primitive wonderful varieties}, preprint arXiv:1106.3187

\bibitem[BP14]{BP14} Bravi, P. and Pezzini, G., \textit{Wonderful subgroups of reductive groups and spherical systems}, J. Algebra \textbf{409} (2014), 101-147.

\bibitem[Bri1]{Br89-1} Brion, M., \textit{On spherical varieties of rank one}, CMS Conf. Proc. \textbf{10} (1989), 31--41.

\bibitem[Bri2]{Br89} Brion, M., \textsl{Groupe de Picard}, Duke Math. Journal, Volume \textbf{58}, Number 2 (1989), 397-424.

\bibitem[Bri3]{Br97} Brion, M., \textit{Vari\'et\'es sph\'eriques}, Notes de la session de la S. M. F. ``Op\'erations hamiltoniennes et op\'erations de groupes alg\'ebriques", Grenoble, 1997, 1--60.

\bibitem[Bri4]{Br07} Brion, M., \textit{The total coordinate ring of a wonderful variety}, J. Algebra \textbf{313} (2007), 61--99.

\bibitem[Bri5]{Br10} Brion, M., \textit{Introduction to actions of algebraic groups}, Notes of the course "Actions hamiltoniennes : invariants et classification" (CIRM, Luminy, 2009), available on-line at http://ccirm.cedram.org/ccirm-bin/feuilleter.

\bibitem[Ca]{Ca} Camus, R., \textit{Vari\'et\'es sph\'eriques affines lisses}, Ph. D. thesis, Institut Fourier, Grenoble, 2001, available at http://www-fourier.ujf-grenoble.fr.

\bibitem[DP]{DCP} De Concini, C. and Procesi, C., \textit{Complete symmetric varieties}, Invariant theory (Montecatini, 1982), Lecture Notes in Math., 996, Springer, Berlin, 1983, 1--44.

\bibitem[CLS]{CLS} Cox, D. A., Little J. B. and Schenck, H., \textit{Toric varieties}, Invariant theory (Montecatini, 1982), Graduate Studies in Math. American Math. Soc., to appear, available at http://www.cs.amherst.edu/~dac/toric.html.


\bibitem[F]{F} Foschi, A., \textit{Vari\'et\'es maginifiques et polytopes moment}, Ph. D. thesis, Institut Fourier, Grenoble, 1998, available at http://www-fourier.ujf-grenoble.fr.

\bibitem[H]{Ho} Hoschild, G., \textit{Cohomology of algebraic linear groups}, Illinois J. Math, \textbf{5} (1961), 492--519.

\bibitem[Js]{Js} Jansou, S., \textit{D\'eformations invariantes des c\^ones de vecteurs primitifs}, Math. Ann. \textbf{338} (2007), 627--647.

\bibitem[Jt]{Jt} Jantzen, J. C., \textit{Representations of algebraic groups}, Mathematical Surveys and Monographs, vol. 107, American Mathematical Society, Providence, RI, 2003.

\bibitem[KR]{KR} Kempf, G. R., and Ramanathan A., \textit{Multi-cones over Schubert varieties}, Invent. math. \textbf{87} (1987), 353--363.


\bibitem[K1]{K91} Knop, F., Proceedings of the Hyderabad Conference on Algebraic Groups, December 1989. Madras: Manoj Prakashan (1991), 225--249.

\bibitem[K2]{K96} Knop, F., \textit{Automorphisms, root systems, and compactifications of homogeneous varieties}, J. Amer. Math. Soc. \textbf{9} (1996),
153--174.

\bibitem[Ko]{Ko} Kostant, B., \textsl{Lie algebra cohomology and generalized Schubert cells}, Ann. of Math. \textbf{77} (1963), 72--144.

\bibitem[Lo]{Lo} Losev, I.V., Uniqueness property for spherical homogeneous spaces, Duke Math. Journal, Volume \textbf{147}, Number 2  (2009), 315--343.

\bibitem[Lu1]{L96} Luna, D., \textit{Toute vari\'et\'e magnifique est sph\'erique}, Transform.\ Groups \textbf{1} (1996), 249--258.

\bibitem[Lu2]{L97} Luna, D., \textit{Grosses cellules pour les vari\'et\'es magnifiques}, in Algebraic Groups and Lie Groups, ed. by G. I. Lehrer, Australian Math. Soc. Lecture, Series 9 (1997), 267-280.

\bibitem[Lu3]{Lu01} Luna, D., \textit{Vari\'et\'es sph\'eriques de type $A$}, Inst.\ Hautes \'Etudes Sci.\ Publ.\ Math.\ \textbf{94} (2001), 161--226.

\bibitem[Lu4]{Lu07} Luna, D., \textit{La vari\'et\'e magnifique mod\`ele}, Journal of Algebra \textbf{313}, (2007), 292--319.

\bibitem[LV]{LV} Luna, D., and Vust, T., \textit{Plongements d'espaces homog\`enes}. Comment. Math. Helv. \textbf{58}, (1983), 186--245.

\bibitem[MF]{MF} Mumford, D., and Fogarty, J., Geometric Invariant Theory, Second enlarged edition, Ergebnisse \textbf{34}, Springer Verlag, 1982.

\bibitem[PVS]{PVS} Papadakis, S., and Van Steirteghem B., \textit{Equivariant degenerations of spherical modules for groups of type $\mathsf A$},
preprint \verb|arXiv:math.1008.0911|.

\bibitem[S]{Se} Sernesi, E., Deformation of algebraic schemes, Grundlehren der mathematischen Wissenchaften \textbf{334},  Springer-Verlag, Berlin, 2006.

\bibitem[W]{W} Wasserman, B., \textit{Wonderful varieties of rank two}, Transform.\ Groups \textbf{1} (1996), 375--403.


\end{thebibliography}
\end{document}